\documentclass[12pt,hyphens,spaces]{amsart}
\usepackage{amsmath, amsthm, amssymb,esint}
\usepackage{fullpage}
\usepackage{color}
\usepackage{hyperref}
\usepackage{cancel}
\usepackage{soul}
\usepackage{enumitem}
\usepackage{tikz}
\usetikzlibrary{positioning}
\usepackage[overload]{empheq}

\usepackage{amscd} 
\usepackage{amsfonts} 
\usepackage{amsmath} 
\usepackage{amsrefs} 
\usepackage{amssymb} 
\usepackage{amsthm} 
\usepackage{latexsym} 
\usepackage{mathrsfs} 
\usepackage{mathtools} 
\usepackage{slashed}  

\usepackage{cases}
\usepackage[dvips]{epsfig}
\usepackage{epsf} 

\usepackage{enumitem} 

\usepackage{pifont} 
\usepackage{upgreek} 
\usepackage{fancyhdr} 
\usepackage{calligra} 
\usepackage{marvosym} 
\usepackage[percent]{overpic} 
\usepackage{pict2e} 
\usepackage[hypcap=false]{caption} 
\usepackage{wasysym} 
\usepackage{accents}

\usepackage[papersize={8.5in,11in}, margin=1in,footskip=0.4in,vmarginratio={1:1},hmarginratio={1:1}]{geometry} 

\theoremstyle{plain}
\newtheorem{theorem}{Theorem}[section]

\newtheorem{lemma}[theorem]{Lemma}
\newtheorem{proposition}[theorem]{Proposition}
\newtheorem{corollary}[theorem]{Corollary}

\newtheorem{remark}[theorem]{Remark}

\newtheorem{definition}[theorem]{Definition}

\numberwithin{equation}{section}
\allowdisplaybreaks

\newcommand{\norm}[1]{\Vert#1\Vert}

\newcommand{\curl}{\mbox{\upshape curl}\mkern 1mu}

\newcommand{\MovingDom}{\Omega}
\newcommand{\MovingDomT}{\Omega_t}

\newcommand{\Dt}{D_t}

\newcommand{\G}{G}
\newcommand{\GInv}{G}
\newcommand{\GInvP}{G}
\newcommand{\Min}{m}

\newcommand{\Vort}{\omega}

\newcommand{\F}{F}

\renewcommand{\H}{{\mathcal H }}
\newcommand{\bfH}{{\mathbf H }}

\newcommand{\rz}{\mathring{r}}
\newcommand{\vz}{\mathring{v}}
\newcommand{\GInvz}{\mathring{G}}

\newcommand{\ro}{\check{r}}
\newcommand{\vo}{\check{v}}
\newcommand{\xo}{\check{x}}

\newcommand{\D}{{\mathcal{D}}}
\newcommand{\tD}{{\tilde{\mathcal{D}}}}

\newcommand{\R}{{\mathbb R}}

\newcommand{\br}{\langle r \rangle}

\begin{document}

\title[Relativistic vacuum boundary Euler]{
The relativistic Euler equations with a physical vacuum boundary: Hadamard local well-posedness,
rough solutions, and continuation criterion}

\author{Marcelo M. Disconzi}
\address{Vanderbilt University, Nashville, TN, USA.}
\email{marcelo.disconzi@vanderbilt.edu}

\author{Mihaela Ifrim}
\address{Department of Mathematics, University of Wisconsin, Madison}
\email{ifrim@wisc.edu}

\author{Daniel Tataru*}
\address{Department of Mathematics, University of California at Berkeley}
\email{tataru@math.berkeley.edu}

\begin{abstract}
In this paper we provide a complete local well-posedness theory for the free boundary relativistic Euler equations with a  physical vacuum boundary on a Minkowski background. 
Specifically, we establish the following results:
(i) local well-posedness in the Hadamard sense, i.e., local
existence, uniqueness, and continuous dependence on the data;
(ii) low regularity solutions: our uniqueness result holds at the level
of Lipschitz velocity and density, while our rough solutions, 
obtained as unique limits of smooth solutions, have regularity 
only a half derivative above scaling;
(iii) stability: our uniqueness in fact follows
from a more general result, namely, we show that a certain
nonlinear functional that tracks the distance between two solutions (in part by measuring the distance between their
respective boundaries) is propagated by the flow;
(iv) we establish sharp, essentially scale invariant energy estimates for solutions; (v) a 
sharp continuation criterion, at the level of scaling, showing
that solutions can be continued as long as the velocity is in $L^1_t Lip$ and
a suitable weighted version of the density is at the same regularity
level.

Our entire approach is in Eulerian coordinates and relies on the functional framework developed in the companion work of the second and third authors on corresponding  non relativistic problem. All our results are valid for a general equation of state $p(\varrho)= \varrho^\gamma$, $\gamma > 1$.

\end{abstract}

\maketitle
\setcounter{tocdepth}{1}
\tableofcontents

\section{Introduction\label{S:Intro}}

In this article, we consider the relativistic Euler equations, which describe the motion of a relativistic  fluid in a Minkowski background $\mathbb M^{d+1}$, $d \geq 1$. The fluid state is represented by the (energy) \emph{density} $\varrho \geq 0$, and the 
\emph{relativistic velocity} $u$. The velocity is assumed to be a forward time-like vector field, normalized 
by 
\begin{align}
u^\alpha u_\alpha = -1.
\label{E:Normalization_u}
\end{align}
The equations of motion consist of
\begin{align}
\partial_\alpha \mathcal{T}^\alpha_\beta = 0,
\label{E:Div_T}
\end{align}
where $\mathcal{T}$ is the energy-momentum tensor for a perfect fluid, defined by
\begin{align}
\mathcal{T}_{\alpha\beta} := (p + \varrho) u_\alpha u_\beta + p \, \Min_{\alpha\beta},
\label{E:Energy_momentum_tensor}
\end{align}
Here $\Min$ is the Minkowski metric, and $p$ is the \emph{pressure}, which is subject to the 
equation of state
\begin{align}
p = p(\varrho) .
\nonumber
\end{align}
Projecting \eqref{E:Div_T} onto the directions parallel and perpendicular to $u$, using  definition
\eqref{E:Energy_momentum_tensor}, and the identity \eqref{E:Normalization_u}, yields the system
\begin{equation}\label{E:Basic}
\left\{ \begin{aligned}
& u^\mu \partial_\mu \varrho + (p + \varrho)\partial_\mu u^\mu = 0
\\
& (p+\varrho) u^\mu \partial_\mu u_\alpha + \Pi^\mu_\alpha \partial_\mu p = 0,
\end{aligned}\right.
\end{equation}
with $u$ satisfying the constraint \eqref{E:Normalization_u}, which is in turn 
preserved by the time evolution.
Here $\Pi$ is is the projection on the space orthogonal to $u$ and is given by 
\begin{align}
\Pi_{\alpha\beta} =  \Min_{\alpha\beta} +  u_\alpha u_\beta.
\nonumber
\end{align}

Throughout this
paper, we adopt standard rectangular coordinates
in Minkowski space, denoted by $\{x^0,x^1,\dots,x^d\}$,
and we identify $x^0$ with a time coordinate, $t:= x^0$. Greek indices vary from $0$ to $d$
and Latin indices from $1$ to $d$.

The system \eqref{E:Basic} can be seen as a nonlinear hyperbolic system, which in the reference frame of the moving fluid
has the propagation speed
\begin{align}
c_s^2(\varrho) := \frac{ d p }{d\varrho},
\nonumber
\end{align}
which is subject to 
\begin{align}
0 \leq c_s < 1,
\nonumber
\end{align}
implying that the speed of propagation of sound waves is always non-negative and below the speed of light (which equals to one in the units we adopted). 

\medskip

In this article we consider the physical situation where vacuum states are allowed, i.e. the density is  allowed to vanish.  The gas is located in the moving domain 
\begin{align}
\MovingDomT :=   \{ x\in \mathbb{R}^{d} \, | \, \varrho(t,x) > 0 \},
\nonumber
\end{align}
whose boundary $\Gamma_t$ is the \emph{vacuum boundary}, which is 
advected by the fluid velocity $u$. 

The distinguishing characteristic of a gas, versus the case of a liquid, is that the density,
and implicitly the pressure and the sound speed, vanish on the free boundary $\Gamma_t$,
\begin{equation}
    \varrho = 0, \quad p = 0, \quad c_s = 0  \qquad \text{ on } \Gamma_t.
    \nonumber
\end{equation}
Thus, the equations studied here
provide a basic model of relativistic gaseous stars (see Section \ref{S:Historical}).
An appropriate equation of state to describe this
situation is\footnote{Observe that the requirement
$0\leq c_s^2 <  1$ imposes a bound on $\varrho$. This occurs
because power-law equations of state such as 
\eqref{E:Equation_of_state_general} are no longer valid
if the density is very large \cite{Kurkela:2014vha}.}
(see, e.g., \cite{RezzollaZanottiBookRelHydro}*{Section 2.4}
or \cite{Rendall-1992}):
\begin{align}
p(\varrho) =  \varrho^{\kappa+1}, \, \text{ where } \kappa > 0
\text{ is a constant.}
\label{E:Equation_of_state_general}
\end{align}

The decay rate of the sound speed at the free boundary plays a critical role. 
Precisely,
there is a unique, natural decay rate which is consistent with the time evolution 
of the free boundary problem for the relativistic Euler gas, which is commonly referred to as
\emph{physical vacuum}, and has the form 
\begin{align}
 c^2_s (t,x) \approx \operatorname{dist}(x,\Gamma_t) \qquad \text{ in } \Omega_t,
 \label{E:Physical_condition_c_bry}
\end{align}
where $\operatorname{dist (\cdot, \cdot)}$ is the distance function.
Exactly the same requirement is present in the non-relativistic compressible Euler equations.
As in the non-relativistic setting,  \eqref{E:Physical_condition_c_bry} should be considered as a condition on the initial data that is propagated by the time-evolution.

\medskip

There are two classical approaches in fluid dynamics, using either Eulerian coordinates,
where the reference frame is fixed and the fluid particles are moving,  or using Lagrangian
coordinates, where the particles are stationary but the frame is moving.  Both of these 
approaches have been extensively developed in the context of the Euler 
equations, where the local well-posedness problem is very well understood. 

By contrast, the free boundary problem corresponding to the physical vacuum  
has been far less studied and understood. Because of the difficulties related
to the need to track the evolution of the free boundary, all the prior work is 
in the Lagrangian setting and in high regularity spaces which are only indirectly defined.

Our goal in this paper is to provide the \emph{first local well-posedness result for this problem.}
Unlike previous approaches, which were limited to proving 
energy-type estimates 
at high regularity and in a Lagrangian setting \cites{Jang-2015,Hadzic-Shkoller-Speck-2019}, here we consider this problem fully within the Eulerian framework, where we provide a complete local well-posedness theory, in the Hadamard sense, in a low regularity setting. We summarize here the main features of our result, which mirror 
the results in the last two authors' prior paper devoted to the non-relativistic problem~\cite{IT-norel}, referring to Section \ref{S:Results} for precise statements:

\begin{enumerate}[label=\alph*)]
\item We prove the \emph{uniqueness} of solutions with very limited regularity\footnote{In an appropriately weighted sense in the case of $\varrho$, see Theorem~\ref{T:Uniqueness}.}  $v \in Lip$, $\varrho \in Lip$. More generally, at the same regularity level we prove \emph{stability}, by showing that bounds for a certain nonlinear distance between different solutions can be propagated in time.

\item Inspired by \cite{IT-norel}, we set up the Eulerian Sobolev \emph{function space structure} where this problem should be considered, providing the correct, natural scale of spaces for this evolution.

\item We prove sharp, \emph{scale invariant\footnote{While this problem does not have 
an exact scaling symmetry, one can still identify a leading order scaling.} energy estimates} within the above mentioned scale of spaces,
which guarantee that the appropriate  Sobolev regularity of solutions can be continued for as long as we have uniform bounds at the same scale  $v \in Lip$. 

\item We give a constructive proof of \emph{existence} for regular solutions, fully within the Eulerian setting, based on  the above energy estimates.

\item We employ a nonlinear Littlewood-Paley type method, developed  prior work \cite{IT-norel}, in order to obtain \emph{rough solutions} as unique limits of smooth solutions. This also yields the \emph{continuous dependence} of the solutions on the initial data.
\end{enumerate}

\subsection{Space-time foliations and the material derivative}
The relativistic character of our problem implies that there
is no preferred choice of coordinates. On the other
hand, in order to derive estimates and make quantitative
assertions about the evolution, we have
to choose a foliation of spacetime by space-like hypersurfaces.
Here, we take advantage of the natural set-up provided
by Minkowski space and foliate the spacetime
by $\{ t = constant \}$ slices. We then define
the material derivative, which is adapted to this
specific foliation, as
\begin{align}
    \Dt := \partial_t + \frac{u^i}{u^0} \partial_i.
    \label{E:Material_derivative_u}
\end{align}
The vectorfield $\Dt$ is better adapted to the study
of the free-boundary evolution than
working directly with $u^\mu \partial_\mu$. 
Indeed, in order to track the motion of fluid particles on the boundary, we need to understand 
their velocity relative to the aforementioned spacetime
foliation. The velocity 
that is measured by an observer in a reference frame characterized by the coordinates $(t,x^1,\dots,x^d)$
is $u^i/u^0$.  This is a consequence of the fact that in relativity observers are defined by their world-lines, which can be 
reparametrized. This ambiguity is fixed by imposing the constraint $u^\mu u_\mu = -1$. As a consequence,
the $d$-dimensional vectorfield $(u^1,\dots,u^d)$ can have norm arbitrarily large, while the physical
velocity has to have norm at most one (the speed of light). 

It follows, in particular, that fluid particles on the boundary move with velocity $u^i/u^0$. These considerations
also imply that the standard differentiation
formula for moving domains holds with $\Dt$, i.e.,
\begin{align}
\frac{d}{dt} \int_{\MovingDomT}  f \, dx
= \int_{\MovingDomT} \Dt  f \, dx+ \int_{ \MovingDomT} f \partial_i \left( \frac{u^i}{u^0} \right)\, dx.
\label{E:Derivative_moving_domain}
\end{align}
This formula remains
valid with the good variable $v$  we introduce
below since $v^i/v^0 = u^i/u^0$.

\subsection{The good variables\label{S:Good_variables}} The starting
point of our analysis is a good choice of dynamical
variables. We seek variables that are tailored
to the characteristics of the Euler flow \emph{all the way to moving boundary,} where the sound characteristics degenerate due to the vanishing of the sound speed. Our choice of good variables will 

\begin{enumerate}[label=(\roman*)]
\item better diagonalize the system with respect to the material derivative, 
\item  be associated with truly relativistic properties of the vorticity, and 
\item lead to good weights that allow us to control the behavior of the fluid variables when
one approaches the boundary. 
\end{enumerate}

Property (i) will be intrinsically tied
with both the wave and transport character
of the flow 
in that (a) the diagonalized 
equations lead to good second order equations
that capture the propagation of sound in the 
fluid, see 
Section \ref{S:Transition_operators}, and (b)
it provides a good transport structure
that will allow us to implement
a time discretization for the construction
of regular solution, see Section \ref{S:Regular_solutions}.
Property (ii)
will ensure a good coupling between
the wave-part and the transport-part of the system.
Finally, property (iii) will lead to the correct
functional framework needed to close the estimates\footnote{It is
well known that we can think of the relativistic Euler flow as a
wave-transport system. What is relevant here is that
the wave evolution that comes out of the
diagonalized equations allows estimates all the way to the 
free surface.}. Our good variables, denoted by $(r,v)$,
are defined in \eqref{E:v_def} and \eqref{E:r_def}.
The corresponding equations of motion
are \eqref{E:r_v_eq}, which we now derive. 

Our first choice of good variables is a rescaled
version of the velocity given by
\begin{align}
v^\alpha = f(\varrho) u^\alpha,
\label{E:v_def}
\end{align} 
where $f$ is given by
\begin{align}
    f(\varrho) := \exp 
    \int \frac{c_s^2(\varrho)}{p(\varrho) + \varrho} \, d\varrho.
      \label{E:f_of_rho_general_def}
\end{align}
Although we are interested in the case 
$p(\varrho) = \varrho^{\kappa+1}$, it is instructive to 
consider first a general barotropic equation
of state; see the discussion related to the 
vorticity further below.

In order to understand our choice for $f$,
compute
\begin{align}
    \partial_\mu v^\alpha = f^\prime(\varrho) \partial_\mu \varrho 
    u^\alpha + f(\varrho) \partial_\mu u^\alpha.
    \nonumber
\end{align}
Solving for $\partial_\mu u^\alpha$ and plugging
the resulting expression into the second equation
of \eqref{E:Basic} we find
\begin{align}
    \frac{p +\varrho}{f} u^\mu \partial_\mu v^\alpha
    + c_s^2 \Min^{\alpha\mu} \partial_\mu \varrho
    + \left(-\frac{f^\prime}{f} (p + \varrho) + 
    c_s^2 \right) u^\alpha u^\mu \partial_\mu \varrho 
    = 0.
    \nonumber
\end{align}
We see that the term in parenthesis vanishes if
$f$ is given by \eqref{E:f_of_rho_general_def}, resulting
in an equation which is diagonal with respect to the material 
derivative, and which we write as
\begin{align}
    \Dt v^\alpha + \frac{c_s^2 f^2}{(p+\varrho) v^0} \Min^{\alpha \mu} \partial_\mu \varrho = 0.
    \label{E:v_eq_general}
\end{align}
We notice that in terms of $v$, the material
derivative \eqref{E:Material_derivative_u} reads
\begin{align}
\Dt & =  \partial_t + \frac{v^i}{v^0} \partial_i.
\nonumber
\end{align}
In view of the constraint \eqref{E:Normalization_u}, 
we have that $v^0$ satisfies
\begin{align}
    v^0 = \sqrt{f^2 + |v|^2}, \quad |v|^2 & := v^i v_i,
    \label{E:v_0_constraint_general}
\end{align}
and in solving for $v^0$ we chose the positive
square root because $u$, and thus $v$, is a future-pointing
vectorfield. 

We now show that our choice \eqref{E:v_def} also diagonalizes the first equation in \eqref{E:Basic}. First,
we use \eqref{E:v_eq_general} with $\alpha=0$
and solve to $\partial_t v^0$, obtaining
\begin{align}
    \begin{split}
        \partial_t v^0 &= 
        \frac{c_s^2 f^2}{(p+\varrho)v^0} \partial_t \varrho
        - \frac{v^i}{v^0} \partial_i v^0
        \\
        &= \frac{c_s^2 f^2}{(p+\varrho)v^0} \partial_t \varrho
        - \frac{f f^\prime}{(v^0)^2} v^i \partial_i \varrho 
        - \frac{v^i v^j}{(v^0)^2} \partial_i v_j,
    \end{split}
    \nonumber
\end{align}
where in the second equality we used \eqref{E:v_0_constraint_general}
to compute $\partial_i v^0$. Using the above 
identity for $\partial_t v^0$, we find the following
expression for $\partial_\mu v^\mu$:
\begin{align}
    \partial_\mu v^\mu 
    = \frac{c_s^2 f^2}{(p+\varrho)v^0} \partial_t \varrho
    - \frac{f f^\prime}{(v^0)^2} v^i \partial_i \varrho
    + \left( \updelta^{ij} - \frac{v^i v^j}{(v^0)^2} \right)
    \partial_i v_j,
    \nonumber
\end{align}
where $\updelta$ is the Euclidean metric.
Expressing $\partial_\mu u^\mu$ in terms of 
$\partial_\mu v^\mu$ (and derivatives of $\varrho$)
and using the above expression for $\partial_\mu v^\mu$,
we see that the first equation in \eqref{E:Basic}
can be written as
\begin{align}
    \Dt \varrho + \frac{p+\varrho}{ a_0 v^0} 
    \left( \updelta^{ij} - \frac{v^i v^j}{(v^0)^2} \right)
    \partial_i v_j
    - c_s^2 \frac{2 f^2 }{a_0 (v^0)^3} v^i \partial_i
    \varrho = 0.
    \label{E:rho_eq_general}
\end{align}
Here we are using the notation
\begin{align}\label{def-a0}
a_0 & := 1 - c_s^2\frac{|v|^2}{(v^0)^2}.
\end{align}

Observe that equations
\eqref{E:v_eq_general} and \eqref{E:rho_eq_general}
are valid for a general barotropic equation
of state. We now assume the equation of state
\eqref{E:Equation_of_state_general}.
Then the sound speed is given by
$c_s^2 = (\kappa+1)\varrho^\kappa$
and $f$ becomes $f(\varrho) = (1+\varrho^\kappa)^{1+\frac{1}{\kappa}}$ (we 
choose the constant of integration by setting $f(0)=1$, so 
that $v=u$ when $\varrho=0$). It turns out
that it is better to adopt the sound speed squared as 
a primary variable instead of $\varrho$ because
it plays the role of the
correct weight in our energy functionals. We thus define\footnote{The factor
$\frac{1}{\kappa}$ in the definition of $r$ is 
a matter of convenience. Although $r$ and $c_s^2$ differ
by this factor, we slightly abuse the terminology
and also call $r$ the sound speed squared.} the second component of our good variables by
\begin{align}
    r:=\frac{1+\kappa}{\kappa} \varrho^\kappa.
    \label{E:r_def}
\end{align}

Therefore, using $(r,v)$ as  our good variables, and  $p(\varrho)$ given by \eqref{E:Equation_of_state_general} we find that 
the equations \eqref{E:v_eq_general} and \eqref{E:rho_eq_general} become
\begin{subequations}{\label{E:r_v_eq}}
\begin{align}[left = \empheqlbrace\,]
  & \Dt r + r \GInvP^{ij} \partial_i v_j  + r a_1 v^i \partial_i r = 0 
  \label{E:r_eq}
  \\
  & \Dt v_i + a_2 \partial_i r  = 0,
  \label{E:v_eq}
\end{align}
\end{subequations}
where we have defined
\begin{equation}
\GInvP^{ij}  := \frac{\kappa \br}{a_0 v^0} \left( \updelta^{ij} - \frac{v^i v^j}{(v^0)^2} \right), \quad \br:= 1+\frac{\kappa r}{\kappa +1}, 
\nonumber
\end{equation}
and the coefficients $a_0,a_1$ and $a_2$ are given by 
\begin{equation}
\nonumber
a_0  := 1 - \kappa r \frac{|v|^2}{(v^0)^2},
\qquad
a_1  := -\frac{2 \kappa \br^{2+\frac{2}{\kappa}} }{(v^0)^3 a_0},
\qquad
a_2  := \frac{\br^{1+\frac{2}{\kappa}}}{v^0}.
\end{equation}

Equations \eqref{E:r_v_eq} are the desired 
diagonal with respect to $\Dt$ equations, and the rest 
of the article will be based on them.
In writing these equations we consider
only the spatial components $v^i$  
as variables, with $v^0$ always given by
\begin{align}
    v^0 = \sqrt{\br^{2+\frac{2}{\kappa}} + |v|^2}.
    \label{E:v_0_constraint}
\end{align}
The specific form 
of the coefficients $a_0$, $a_1$, and $a_2$ is not
very important for our argument. We essentially only
use that they are smooth functions of $r$ and $v$, and
that $a_0, \, a_2>0$.

The operator $\GInvP^{ij} \partial_i (\cdot)_j$ can be viewed 
as a divergence type operator. This divergence structure is related to the
fact that equations \eqref{E:r_v_eq}
express the wave-like behavior of $r$ and
of the
divergence part of $v$. The symmetric  
and positive-definite matrix
$c_s^2 \GInvP^{ij}$ is closely related to the inverse 
of the acoustical metric; precisely, they agree 
at the leading order near the boundary.

As we will see, equations \eqref{E:r_v_eq}
also have the correct balance of powers
of $r$ to allow estimates all the way to the
free boundary. The $r$ factor in the divergence
of $v$ is related to the propagation
of sound in the fluid (see Section
\ref{S:Transition_operators}) whereas the
$r$ factor in the last term 
of \eqref{E:r_eq} will allow us to treat 
$r a_1 v^i \partial_i r$ essentially as
a perturbation at least in elliptic estimates
(see Section \ref{S:Energy_estimates}).

One can always diagonalize equations
\eqref{E:Basic} by simply algebraically 
solving for 
$\partial_t (\varrho,u)$. But it is not difficult
to see that this procedure will not lead
to equations with good structures for the study
of the vacuum boundary problem. In this regard, 
observe that the choice \eqref{E:v_def}
is a nonlinear change of variables, whereas
algebraically solving for $\partial_t (\varrho,u)$
is a linear procedure. 

We now comment on the relation between
$v$ and the vorticity of the fluid $\omega$.
It is well-known (see, e.g.,
\cite{Choquet-BruhatBook}*{Section IX.10.1})
that in relativity
the correct notion of vorticity
is given by the following two-form in spacetime
\begin{align}
\begin{split}
\Vort_{\alpha \beta} & := (d_{st} v )_{\alpha\beta}  = \partial_\alpha v_\beta - \partial_\beta  v_\alpha,
\end{split}
\label{E:Relativistic_vort_def}
\end{align}
where $d_{st}$ is the exterior derivative in spacetime. This is true 
not only for the power law equation of state \eqref{E:Equation_of_state_general},
but also for an arbitrary barotropic equation of state.

A computation using 
\eqref{E:Relativistic_vort_def} 
(see, e.g., \cite{Choquet-BruhatBook})
and the equations of motion implies
that
\begin{align}
v^\alpha \Vort_{\alpha \beta} = 0,
\label{E:Contraction_vort_V}
\end{align}
and that $\Vort$ satisfies the following evolution equation
\begin{align}
v^\mu \partial_\mu \Vort_{\alpha\beta} + \partial_\alpha v^\mu \Vort_{\mu \beta} 
+ \partial_\beta v^\mu \Vort_{\alpha \mu } = 0.
\label{E:Relativistic_vorticity_evolution}
\end{align}
Observe that \eqref{E:Relativistic_vorticity_evolution} implies that
$\Vort = 0$ if it vanishes initially.

Since we will consider only the spatial components of $v$ as independent, we use 
\eqref{E:Contraction_vort_V} to eliminate the 
$0j$ components of $\Vort$ from \eqref{E:Relativistic_vorticity_evolution} as follows.
From \eqref{E:Contraction_vort_V} we can write
\begin{align}
\Vort_{0 j} = -\frac{v^i}{v^0} \Vort_{i j}.
\label{E:Vort_0i_spatial}
\end{align}
Using \eqref{E:Vort_0i_spatial} into \eqref{E:Relativistic_vorticity_evolution} 
with $\alpha,\beta=i,j$ we finally obtain
\begin{align}
\Dt \Vort_{ij} 
+ \frac{1}{v^0} \partial_i v^k \Vort_{k j} 
+ \frac{1}{v^0}  \partial_j v^k \Vort_{i k } 
- \frac{1}{(v^0)^2}  \partial_iv^0v^k \Vort_{kj}
+ \frac{1}{(v^0)^2}  \partial_j v^0 v^k\Vort_{ki}
= 0.
\label{E:Transport_vorticity}
\end{align}
Equation \eqref{E:Transport_vorticity} will be used to derive estimates for $\Vort_{ij}$ that will
complement the estimates for $r$ and the
divergence of $v$ obtained from \eqref{E:r_v_eq}.

We remark that in the literature, the use of $v$, given
by \eqref{E:v_def}, seems
to be restricted mostly to definition and 
evolution of the vorticity. To the best of our
knowledge this is the first time that it was
observed that the same change of variables needed
to define the relativistic vorticity also
diagonalizes the equations of motion with respect
to $\Dt$.

\subsection{Scaling and bookkeeping scheme\label{S:Bookkeeping}}
Although equations \eqref{E:r_v_eq}
do not obey a scaling law, it is still possible
to identify a scaling law for the leading
order dynamics near the boundary. This will 
motivate the control
norms we introduce in the next section, as well 
as provide a bookkeeping scheme that will allow us
to streamline the analysis of many complex multilinear
expressions we will encounter.

As we will see, the contribution of last term in \eqref{E:r_eq} to our energies is negligible, due to the multiplicative $r$ factor. Thus, we ignore this term
for our scaling analysis\footnote{ And then it indeed turns out to be lower order 
from a scaling perspective.}.
Replacing all
coefficients that are functions of $(r,v)$ by $1$, while keeping the transport
and divergence structure present in the equations, we obtain
the following simplified version of \eqref{E:r_v_eq}:
\begin{subequations}{\label{E:Leading_order_eq_scaling}}
\begin{align}[left = \empheqlbrace\,]
& ( \partial_t + v^j \partial_j) r + r \updelta^{ij} \partial_i v_j \sim 0
\label{E:Leading_order_r}
\\
& (\partial_t + v^j \partial_j)  v_i + \partial_i r \sim 0.
\label{E:Leading_order_v}
\end{align}
\end{subequations}
This system is expected to capture the leading order dynamics near the boundary, and also mirrors
the nonrelativistic version of the compressible Euler equations, considered in the predecessor 
 to this paper, see \cite{IT-norel}.
Equations \eqref{E:Leading_order_eq_scaling} admit the scaling law
\begin{align}
(r(t,x),v(t,x)) \mapsto (\lambda^{-2} r(\lambda t, \lambda^2 x), \lambda^{-1} v(\lambda t, \lambda^2 x) ).
\nonumber
\end{align}
Based on this leading order scaling analysis, we assign the following order to the variables and operators in equations \eqref{E:r_v_eq}:
\begin{enumerate}[label=(\roman*)]
\item $r$ and $v$ have order $-1$ and $-1/2$, respectively.
More precisely, we only count $v$ as having order $-1/2$ when it is differentiated. Undifferentiated $v$'s have order zero.
\item $\Dt$ and $\partial_i$ have order $1/2$ and $1$, respectively.
\item $\GInv$, $a_0$, $a_1$, and $a_2$, and more generally, any smooth function of $(r,v)$
not vanishing at $r = 0$, have order $0$. 
\end{enumerate}
Expanding on (iii) above, the order of a function of $r$ is defined
by the order of its leading term in the Taylor expansion about $r=0$, being of order zero if
this leading term is a constant. The order of a multilinear expression is defined as the sum of the orders of each factor. Here we remark that all expressions arising in this paper
are multilinear expressions, with the possible exception of nonlinear factors as in (iii) above.

According to this convention, all terms in equation \eqref{E:v_eq} have order zero, and all terms in \eqref{E:r_eq} have order $-1/2$, except for the last term in 
\eqref{E:r_eq} which has order $-1$. Upon successive differentiation of any multilinear 
expression with respect to $\Dt$ or $\partial$, all terms produced are the same (highest)
order, unless some of these derivatives apply to nonlinear factors as in (iii);
then lower order terms are produced.

\subsection{Energies, function spaces, and control norms\label{S:Energies_and_control_norms}}
Here we introduce the function spaces and control norms that we
need in order to state our main results.  A more detailed discussion  is given in Section \ref{S:Function_spaces}. With some obvious adjustments, here we follow the lat two authors' prior work in \cite{IT-norel}.
We assume throughout that $r$ is 
a positive function on $\MovingDomT$, vanishing
simply on the boundary, and so that 
$r$ is comparable to the distance to the boundary $\Gamma_t$.

In order to identify the correct functional 
framework for our problem, we start with the linearization
of the equations \eqref{E:r_v_eq}. In Section \ref{S:Linearized}
we show that the linearized equations admit the following
energy
\begin{align}
    \norm{(s,w)}^2_{\mathcal{H}} = 
    \int_{\MovingDomT}     r^{\frac{1-\kappa}{\kappa}}  ( s^2  
+ a_2^{-1} r \GInvP^{ij} w_i w_j )\, dx,
\nonumber
\end{align}
which defines the (time dependent) weighted $L^2$ space $\H$.

The motivation for the definition of higher order norms and spaces
comes from the good second order equations mentioned in Section \ref{S:Good_variables}.
From equations \eqref{E:r_v_eq}, we find that
the second order evolution is governed at leading order
by a wave-like operator which is essentially a variable coefficient version of 
$\Dt^2 - r \Delta$. This  points toward higher order spaces built on powers
of $r\Delta$. Taking into account also the form
of the linearized energy above, we are led to the following.
We define $\mathcal{H}^{2k}$ as the space of pairs of functions $(s,w)$ in $\MovingDom_t$
for which the norm below is finite
\begin{align}
\norm{(s,w)}^2_{\mathcal{H}^{2k}} := 
\sum_{\vert \alpha \vert =0}^{2k} 
\sum_{\substack{a=0 \\
|\alpha|- a\leq k}}^{k} \norm{ r^{\frac{1-\kappa}{2\kappa} + a} 
\partial^\alpha  s}^2_{L^2}
+
\sum_{\vert \alpha \vert =0}^{2k} 
\sum_{\substack{a=0 \\
|\alpha|- a\leq k}}^{k}
\norm{ r^{\frac{1-\kappa}{2\kappa}+\frac{1}{2} + a} \partial^\alpha w }^2_{L^2}.
\nonumber
\end{align}
The definition of $\mathcal{H}^{2k}$ for non-integer
$k$ is given in Section \ref{S:Function_spaces}, via interpolation.

In view of the scaling analysis of Section
\ref{S:Bookkeeping}, we introduce the 
critical space $\mathcal{H}^{2k_0}$ where
\begin{align}
    2k_0 = d + 1 + \frac{1}{\kappa},
    \label{E:Critical_exponent}
\end{align}
which has the property that its leading order homogeneous component
is invariant with respect to the scaling 
discussed in Section \ref{S:Bookkeeping}. Associated with the exponent $2 k_0$
we define the following scale invariant
time dependent
control norm
\begin{align}
    A := \norm{ \nabla r - N }_{L^\infty}
    + \norm{v}_{\dot{C}^\frac{1}{2}}.
    \nonumber
\end{align}
Here, $N$ is a given non-zero vectorfield with 
the following property. In each sufficiently
small neighborhood of the boundary,
there exists a $x_0 \in \Gamma_t$ such that 
$N(x_0) = \nabla r(x_0)$. The fact that we can
choose such a $N$ follows from the properties 
of $r$. The motivation for introducing
$N$ is that we can make $A$ small by working in small neighborhood
of each reference point $x_0$, whereas $\norm{\nabla r}_{L^\infty}$ is a scale invariant
quantity that cannot be made small by localization
arguments.

We further introduce a second time dependent 
control norm that is associated with 
$H^{2k_0 + 1}$, given by\footnote{ In \cite{IT-norel} the $A$ component is omitted, and $B$
is a homogeneous norm. But here, we need to also add the lower order component $A$
in order to be able to handle lower order terms.}

\begin{align}
    B := A+ \norm{ \nabla r}_{\tilde{C}^\frac{1}{2}}
    + \norm{\nabla v}_{L^\infty},
    \nonumber
\end{align}
where
\begin{align}
    \norm{f}_{\tilde{C}^\frac{1}{2}}
    := \sup_{x,y \in \MovingDomT}
    \frac{ |f(x) - f(y)|}{
    r(x)^\frac{1}{2} + r(y)^\frac{1}{2} +
    |x-y|^\frac{1}{2}}.
    \nonumber
\end{align}
It follows that 
$\norm{ \nabla r}_{\tilde{C}^\frac{1}{2}}$ scales
like the $\dot{C}^\frac{3}{2}$ norm of $r$, but 
it is weaker in that it only uses one derivative
of $r$ away from the boundary.
The norm $B$ will control the growth of our energies, allowing for a secondary dependence
on $A$.

\bigskip

When the density is bounded away from zero,  
the relativistic Euler equations can be written as a first-order symmetric
hyperbolic system (see, e.g., \cite{AnileBook}) and standard techniques can be applied to derive 
local estimates. The difficulties in our case come from the vanishing of $r$
on the boundary. Using the finite speed
of propagation of the Euler flow, we can use a partition of unity to separate the near-boundary
behavior, where $r$ approaches zero, from the bulk dynamics, where
$r$ is bounded away from zero. Furthermore, we can also localize to a small set where
$A$ is small. Such a localization will be implicitly assumed in all our analysis, in order 
to avoid cumbersome localization weights through the proofs.

\subsection{The main results\label{S:Results}}
Here we state our main results. Combined, these
results establish the sharp local well-posedness
and continuation criterion discussed earlier.
We will make all our statements for the system written in terms
of the good variables $(r,v)$, i.e., equations
\eqref{E:r_v_eq}. Readers interested 
in the evolution of \eqref{E:Basic} 
should have no difficulty translating our statements to the original variables 
 $\varrho$ and $u$.

We recall that equations \eqref{E:r_v_eq} are always
considered in the moving domain given by
\begin{align}
    \MovingDom := \bigcup_{0\leq t <T} \{ t \} 
    \times \MovingDomT,
    \nonumber
\end{align}
for some $T>0$, where the moving domain at time $t$, $\MovingDomT$, is given
by 
\begin{align}
    \MovingDomT = \{ x \in \mathbb{R}^d \, | \, 
    r(t,x) > 0 \}.
    \nonumber
\end{align}
We also recall that we are interested in solutions
satisfying the physical boundary condition
\begin{align}
    r(t,x) \approx \operatorname{dist}(x,\Gamma_t),
    \label{E:Physical_boundary_condition}
\end{align}
where $\operatorname{dist} (\cdot, \cdot)$ is the distance function.
Hence, by a \emph{solution} we will always mean a pair of functions $(r,v)$ that satisfies equations \eqref{E:r_v_eq}
within $\MovingDom$, and for which \eqref{E:Physical_boundary_condition} holds.

We begin with our uniqueness result:

\begin{theorem}[Uniqueness]
Equations \eqref{E:r_v_eq} admit at most one solution $(r,v)$
in the class
\begin{align}
    v \in C^1_x, \, \nabla r \in \tilde{C}^{\frac{1}{2}}_x.
    \nonumber
\end{align}
\label{T:Uniqueness}
\end{theorem}

For the next Theorem, we introduce the phase space
\begin{align}
    \mathbf{H}^{2k} := \{ (r,v) \, | \, (r,v) 
    \in \mathcal{H}^{2k} \}.
\end{align}
We refer to Section \ref{S:Function_spaces} for a more
precise definition of $\mathbf{H}^{2k}$, including its topology.
Since the $\mathcal{H}^{2k}$ norms depend on $r$, it is appropriate to think of $\mathbf{H}^{2k}$ in a nonlinear
fashion, as an infinite dimensional manifold.
We also stress that, while
$k$ was an integer in our preliminary discussion in Section \ref{S:Energies_and_control_norms}, in Section
\ref{S:Function_spaces} we extend their definition for 
any $k \geq 0$. Consequently, $\mathbf{H}^{2k}$ is also 
defined for any $k \geq 0$, and our Theorems \ref{T:LWP}
and \ref{T:Continuation} below
include non-integer values of $k$.

\begin{theorem}
Equations \eqref{E:r_v_eq} are locally well-posed
in $\mathbf{H}^{2k}$ for any data $(\mathring{r},\mathring{v}) \in \mathbf{H}^{2k}$ with $\mathring{r}$ satisfying 
\eqref{E:Physical_boundary_condition}, provided that
\begin{align}
    2k > 2 k_0 + 1,
    \label{E:LWP_threshold}
\end{align}
where $k_0$ is given by \eqref{E:Critical_exponent}. 
\label{T:LWP}
\end{theorem}
Local well-posedness in Theorem \ref{T:LWP} is understood
in the usual quasilinear fashion, namely:
\begin{itemize}
    \item Existence of solutions $(r,v) \in C([0,T],\mathbf{H}^{2k})$.
    \item Uniqueness of solutions in a larger class,
    see Theorem \ref{T:Uniqueness}.
    \item Continous dependence of solutions on the initial 
    data in the $\mathbf{H}^{2k}$ topology.
\end{itemize}
Furthermore, in our proof of uniqueness in Section 
\ref{S:Uniqueness} we establish something stronger, namely,
that a suitable nonlinear distance between two solutions
is propagated under the flow. This distance functional, in particular, tracks the distance between the boundaries of the moving domains associated with different solutions. Thus,
our local well-posedness also includes:
\begin{itemize}
    \item Weak Lipschitz dependence on the initial data
    relative to a suitable nonlinear functional introduced
    in Section \ref{S:Uniqueness}.
\end{itemize}

An important threshold for our results corresponds to the 
uniform control parameters $A$ and $B$. 
Of these $A$ is at scaling,
while $B$ is one half of a derivative above scaling. Thus, by 
Lemma~\ref{l:morrey} of Section \ref{S:Function_spaces},
we will have the bounds
\begin{equation}\label{sobolev-A}
A \lesssim \| (r,v)\|_{\bfH^{2k}}, \qquad k > k_0 = \frac{d+1}2 +\frac{1}{2\kappa},    
\nonumber
\end{equation}
and
\begin{equation}\label{sobolev-B}
\qquad B \lesssim \| (r,v)\|_{\bfH^{2k}}, \qquad k > k_0+\frac12 = \frac{d+2}2 +\frac{1}{2\kappa}.
\nonumber
\end{equation}

\bigskip

Next, we turn our attention to the continuation of solutions.
\begin{theorem}
For each integer $k\geq 0$ there exists an energy
functional $E^{2k} = E^{2k}(r,v)$ with the following properties:

\medskip

\noindent a) Coercivity: as long as $A$ remains bounded,
we have
\begin{align}
    E^{2k}(r,v) \approx \norm{(r,v)}^2_{\mathcal{H}^{2k}}.
    \nonumber
\end{align}
\noindent b) Energy estimates hold for solutions 
to \eqref{E:r_v_eq}, i.e.
\begin{align}
    \frac{d}{dt} E^{2k}(r,v) \lesssim_A B \norm{(r,v)}^2_{\mathcal{H}^{2k}}.
    \nonumber
\end{align}
\label{T:Energy_estimates}
\end{theorem}
By Gronwall's inequality, Theorem~\ref{T:Energy_estimates} readily implies
\begin{align}
    \norm{(r,v)(t)}^2_{\mathcal{H}^{2k}} 
    \lesssim e^{\int_0^t C(A) B(\tau) \, d\tau}
    \norm{(\rz,\vz)}_{\mathcal{H}^{2k}}^2,
    \label{E:Energy_bound_from_theorem}
\end{align}
where $C(A)$ is a constant depending on $A$.
The energies $E^{2k}$ will be constructed explicitly only
for integer $k$. Nevertheless, our analysis will show
that \eqref{E:Energy_bound_from_theorem} will also hold
for any $k>0$. This will be done using a mechanism
akin to a paradifferential expansion, without explicitly
constructing energy functionals for non-integer $k$. 
As a consequence, we will obtain:
\begin{theorem}
Let $k$ be as in \eqref{E:LWP_threshold}. Then, the $\mathbf{H}^{2k}$ solution given by Theorem \ref{T:LWP}
can be continued as long as $A$ remains bounded and
$B \in L^1_t(\MovingDom)$.
\label{T:Continuation}
\end{theorem}

\subsection{Historical comments\label{S:Historical}} 
The study of the relativistic Euler equations
goes back to the early days of relativity theory,
with the works of Einstein \cite{Einstein:1914bx}
and Schwarzschild \cite{Schwarzschild-1916}. 
The relativistic free-boundary\footnote{In order
to provide some context, we briefly discuss the general
relativistic free-boundary Euler equations, i.e., including
both the cases of a gas and a liquid. We do not, however,
make an overview of related works that treat
the non-relativistic free-boundary Euler equations. 
See \cite{IT-norel} and references therein for such
a discussion.}
Euler equations were introduced in the '30s in the classical 
works of Tolman, Oppenheimer, and Volkoff
\cites{Tolman:1934za,Tolman:1939jz,Oppenheimer:1939ne},
where they derived the now-called TOV equations.
With the goal of modeling a star in the
framework of relativity, 
Tolman, Oppenheimer, and Volkoff studied spherically symmetric static
solutions to the Einstein-Euler system for
a fluid body in vacuum and identified
the vanishing of the pressure as the correct
physical condition on the boundary.
Observe that such a condition covers
both the cases of a liquid, where $\varrho >0$ on the boundary,
as well as a gas, which we study here, where $\varrho=0$ on the
boundary. This distinction is related to the choice
of equation of state.

Although the TOV equations have a long
history and the study of relativistic stars
is an active and important field of research
(see, e.g., \cite{RezzollaZanottiBookRelHydro}*{Part III}
and \cite{Misner:1974qy}*{Part V}), the mathematical
theory of the relativistic free-boundary Euler equations
lagged behind. 

If we restrict ourselves to spherically-symmetric
solutions, possibly also considering coupling to Einstein's
equations, a few precise and satisfactory mathematical
statements can be obtained.
Lindblom \cite{Lindblom-1988}
proved that a 
static, asymptotically flat spacetime, that contains  only a uniform-density perfect fluid confined to a spatially compact
region ought to be spherically symmetric, thus
generalizing to relativity a classical result 
of Carleman \cite{Carleman-1019} and Lichtenstein \cite{Lichtenstein-1918} for Newtonian fluids.
The proof of existence of spherically symmetric
static solutions to the Einstein-Euler system
consisting of a fluid region and possibly a vacuum region
was obtained by Rendall and Schmidt \cite{Rendall-Schmidt-1991}. Their solutions
allow for the vanishing of the density along
the interface of the fluid-vacuum region, although 
it is also possible that the fluid occupies the entire space
and the density merely approaches zero at infinity.
Makino \cite{Makino-1998} refined this result by 
providing a general criterion for the
equation of state which ensures that the model has finite radius.
Makino has also obtained solutions to the Einstein-Euler equations
in spherical symmetry with a vacuum boundary and near equilibrium
in \cites{Makino-2016,Makino-2017}, where equilibrium here
corresponds to the states given by the TOV equations. In
\cite{Makino-2018}, Makino
extended these results to axisymmetric
solutions that are slowly rotating, i.e., 
when the speed of light is sufficiently large or when the gravitational field is sufficiently weak (see also
the follow-up works \cites{Makino-2019-arxiv,Makino-2019-2-arxiv}
and the preceding work in \cite{Heilig-1995}).
Another result within symmetry class related to the existence
of vacuum regions and relevant for the mathematical
study of star evolution is Had\v{z}i\'{c} and Lin's
recent proof of the ``turning point principle"
for relativistic stars 
\cite{Hadzic-Lin-2021}.

The discussion of the last paragraph was not intended to be 
an exhaustive account of the study of the relativistic
free-boundary Euler equations under symmetry assumptions,
and we refer the reader to the above references 
for further discussion.
Rather, the goal was to highlight that a fair
amount of results can be obtained in symmetry classes.
This is essentially because some of the most challenging aspects
of the problem are absent or significantly simplified
when symmetry is assumed. This should be contrasted with
what is currently known in the general case, 
which we now discuss.

Local existence and uniqueness of solutions the relativistic
Euler equations in Minkowski background
with a compactly supported density have been
obtained by Makino and Ukai \cites{Makino-Ukai-1995, Makino-Ukai-1995-II} and LeFloch and Ukai
\cite{LeFloch-Ukai-2009}. These solutions, however,
require some strong regularity of the fluid variables
near the free boundary and, in particular, do not allow
for the existence of physical vacuum states.
Similarly, Rendall \cite{Rendall-1992}
established a local existence and uniqueness\footnote{More precisely,
only a type of partial uniqueness has been obtained, see
the discussion in \cite{Rendall-1992}.}
result for the Einstein-Euler
system where the density is allowed to vanish. Nevertheless,
as the author himself pointed out, the solutions
obtained are not allowed to accelerate on the free boundary and,
in particular, do not include the physical vacuum case.
Rendall's result has been improved by Brauer and Karp
\cites{Brauer-Karp-2011, Brauer-Karp-2014}, but still
without allowing for a physical vacuum boundary.
Oliynyk \cite{Oliynyk-2012} 
was able to construct solutions 
that can accelerate on the boundary, but 
his result is valid only in one spatial dimension.
A new approach to investigate the free-boundary Euler
equations, based on a frame formalism,
has been proposed by Friedrich in 
\cite{Friedrich-1998} (see also \cite{Friedrich-Rendall-2000})
and further investigated by the first author in
\cite{DisconziRemarksEinsteinEuler}, but it has
not led to a local well-posedness theory.

In the case of a \emph{liquid}, i.e., where the fluid
has a free-boundary where the pressure vanishes but the density
remains strictly positive, a-priori estimates have
been obtained by Ginsberg \cite{Ginsberg-2019} and
Oliynyk \cite{Oliynyk-2017}. Local existence
of solutions was recently established by Oliynyk
\cite{Oliynyk-2019-arxiv} 
whereas Miao, Shahshahani, and Wu 
\cite{Miao-Shahshahani-Wu-2020-arxiv}
proved local existence and uniqueness for the case when
the fluid is in the so-called hard phase, i.e., when 
the speed of sound equals to one.
See also \cite{Trakhinin-2009}, where the author,
after providing a proof of local existence for the
non-isentropic compressible free-boundary Euler equations
in the case of a liquid, discusses ideas to adapt his
proof for the relativistic case.

Finally, for the case treated in this paper, i.e.,
the relativistic Euler equations with a 
\emph{physical vacuum} boundary, the only results
we are aware of are the a-priori estimates
by Had\v{z}i\'{c}, Shkoller, and Speck
\cite{Hadzic-Shkoller-Speck-2019} and 
Jang, LeFloch, and Masmoudi \cite{Jang-LeFloch-Masmoudi-2016}.
In particular, no local existence and uniqueness 
(let along a complete local well-posedness theory 
as we present here) had been previously established.

\subsection{Outline of the paper} 
Our approach
carefully considers the dual
role of $r$, on the one hand, 
as a dynamical variable in the evolution and,
on the other hand, as a defining function
of the domain that, in particular, plays the
role of a weight in our energies.
An important aspect of our approach
is to decouple these two roles.
Such decoupling is what allows
us to work entirely in Eulerian coordinates.
When comparing different solutions 
(which
in general will be defined in different domains),
we can think of the role of $r$ as a defining 
function as leading to a measure of
the distance between the two domains (i.e.,
a distance between the two boundaries),
whereas the role of $r$ as a dynamical variable
leads to a comparison in the common region
defined by the intersection of the two domains.
For instance, in our regularization procedure
for the construction of regular
solutions, the defining functions of the domains are regularized at a different scale than the main dynamical variables.

Although the relativistic and non-relativistic Euler equations,
and their corresponding physical vacuum dynamics, are very 
different, some of our arguments here will closely
follow those in the last two authors' prior work \cite{IT-norel},
where results similar to those of Section \ref{S:Results}
were established for the non-relativistic Euler equations
in physical vacuum. Thus, when it is appropriate, 
we will provide a brief proof, or quote directly from
\cite{IT-norel}. This is particularly the case for Sections
\ref{S:Regular_solutions} and \ref{S:Rough_solutions_and_continuation}.

The paper is organized as follows:

\subsubsection{Function spaces, Section \ref{S:Function_spaces}}
This section presents the functional framework needed
to study equations \eqref{E:r_v_eq}. These are spaces naturally associated with the degenerate wave operator
\begin{align}
    \Dt^2 - r \GInvP^{ij} \partial_i \partial_j
    \nonumber
\end{align}
that is key to our analysis. Similar scales of spaces
have been introduced in  \cite{IT-norel}
treating the non-relativistic case and also in
\cite{Jang-2015} where the non-relativistic problem
had been considered in Lagrangian coordinates and in 
high regularity
spaces.

Our function spaces $\mathcal{H}^{2k}$ are Sobolev-type spaces with weights $r$.
Since the fluid domain is determined by $\Omega_t:= \left\{ r>0\right\}$, the
state space $\mathbf{H}^{2k}$ is nonlinear, having a structure akin to an infinite dimensional manifold.

Interpolation plays two key roles in our work. Firstly,
it allows us to define $\mathcal{H}^{2k}$ for non-integer
$k$ without requiring us to establish direct energy estimates
with fractional derivatives. This is in particular important
for our low regularity setting since the critical exponent
\eqref{E:Critical_exponent} will in general not be an integer.
Secondly, we interpolate between $\mathcal{H}^{2k}$ and 
the control norms $A$ and $B$. For this we use some
sharp interpolation inequalities presented in Section 
\ref{S:Embedding_and_interpolation}. These inequalities
are proven in the last two authors' prior work \cite{IT-norel} and,
to the best of our knowledge, have not appeared in the
literature before. In fact, it is the use of these inequalities
that allows us to work at low regularity, to obtain
sharp energy estimates, and a continuation criterion
at the level of scaling.

\subsubsection{The linearized equation and the
corresponding transition operators, Section \ref{S:Linearized}}
The linearized equation and its analysis form the foundation
of our work, rather than direct nonlinear energy estimates.
Besides allowing us to prove nonlinear energy estimates
for single solutions, basing our analysis on the
linearized equation will also allow us to get good
quantitative estimates for the difference of two solutions.
The latter is important for our uniqueness result and for the 
construction of rough solutions as limits of smooth solutions.
We observe that there are no boundary conditions that 
need to be imposed on the linearized variables. This is related
to the aforementioned decoupling of the roles of $r$ and signals
a good choice of functional framework.

Using the linearized equation we obtain transition operators
$L_1$ and $L_2$ that act at the level of the linearized variables
$s$ and $w$. These transition operators are roughly the leading 
elliptic part of the wave equations for $s$ and the divergence
part of $w$. Note that since the wave evolution for the fluid
degenerates on the boundary due to the vanishing of the sound speed,
so do the transition operators $L_1$ and $L_2$. We refer to $L_1$
and $L_2$ as transition operators because they relate the spaces
$\mathcal{H}^{2k+2}$ and $\mathcal{H}^{2k}$ in a coercive, invertible manner. Because of that, these operators play an important
role in our regularization scheme used to construct 
high-regularity solutions.

\subsubsection{Difference estimates and uniqueness, Section \ref{S:Uniqueness}} In this section we construct 
a nonlinear functional that allows us to measure the distance
between two solutions. 
We show that bounds for  this functional are propagated 
by the flow, which in particular implies 
uniqueness.
A fundamental difficulty is that, since
we are working in Eulerian coordinates, different solutions
are defined in different domains. This difficulty is 
reflected in the nonlinear character of our functional,
which could be thought of as measuring the distance between the boundaries
of two different solutions. The low regularity
at which we aim to establish uniqueness
leads to some technical complications
that are dealt with by a careful analysis
of the problem.

\subsubsection{Energy estimates and coercivity, Section
\ref{S:Energy_estimates}} The energies that we use contain
two components, a wave component and a transport component, in 
accordance with the wave-transport character of the system.
The energy is constructed after identifying Alinhac-type
``good variables" that can be traced back to the 
structure of the linearized problem. This connection
with the linearized problem is also key to establish 
the coercivity of the energy in that it relies on the
transition operators $L_1$ and $L_2$ mentioned above.

\subsubsection{Existence of regular solutions, Section \ref{S:Regular_solutions}}
This section establishes the existence of regular solutions. It heavily relies 
on  the last two authors' prior work \cite{IT-norel}, to which 
the reader is referred for several technical points.

Our construction is based essentially on a 
Newton scheme to produce good
approximating solutions. Nevertheless, a direct
implementation of Newton's method loses derivatives.
We overcome this by preceding each iteration
with a regularization at an appropriate scale
and a separate transport step. The main difficulty
is to control the growth of the energies at 
each step.

\subsubsection{Rough solutions as limits of 
regular solutions, Section \ref{S:Rough_solutions_and_continuation}}
In this section we construct rough
solutions as limits of smooth solutions, in particular establishing the existence
part of Theorem \ref{T:LWP}.
We construct a family of dyadic regularizations
of the data, and control the corresponding
solutions in higher $\H^{2k}$ norms with our energy 
estimates, and the difference of solutions
in $\mathcal{H}$ with our nonlinear stability bounds. 
The latter allow us to establish the convergence of the
smooth solutions to the desired rough solution in weaker topologies.
Convergence in $\bfH^{2k}$ 
is obtained with more accurate control
using frequency envelopes. 
A similar argument then also gives continuous dependence on the data.

\subsection{Notation for $v,\Vort$ and the use of Latin indices}
In view of equations \eqref{E:r_v_eq} and the corresponding
vorticity evolution \eqref{E:Transport_vorticity}, we have  now written 
the dynamics solely in terms of $r$ and the spatial components of $v$, i.e., $v^i$.
We henceforth consider $v$ as a $d$-dimensional vector field, so that whenever referring
to $v$ we always mean $(v^1,\dots,v^d)$. $v^0$ is always understood as a
shorthand for the RHS of \eqref{E:v_0_constraint}. Similarly, by $\Vort$ will stand for $\Vort_{ij}$.

Recalling that indices are raised and lowered with the Minkowski metric and that $\Min_{0i}= 0 = \Min^{0i}$,
$\Min_{ij} = \updelta_{ij}$,
we see that tensors containing only Latin indices have indices equivalently raised and lowered with the 
Euclidean metric. 

\subsection{Acknowledgements} The first author was partially supported by NSF grant DMS-2107701, a Sloan Research Fellowship
provided by the Alfred P. Sloan foundation, a Discovery grant administered by Vanderbilt University, and a
Dean’s Faculty Fellowship. The second author was supported by a Luce Assistant Professorship, by the Sloan Foundation, and by an NSF CAREER grant DMS-1845037. The last author was supported by the NSF grant DMS-1800294 as well as by a Simons Investigator grant from the Simons Foundation.

\section{Function spaces\label{S:Function_spaces}}
Here we define the function spaces that will play a role in our analysis. They are
weighted spaces with weights given by the sound speed 
squared
$r$ which, in view of \eqref{E:Physical_boundary_condition}, is comparable to the distance
to the boundary. More precisely, since a solution to
\eqref{E:r_v_eq} is not a-priori given, in the 
definitions below we take $r$ to be a fixed
non-degenerate defining function for the domain $\MovingDomT$, i.e., proportional to the distance to
the boundary $\Gamma_t$. In turn, the boundary $\Gamma_t$ is assumed 
to be Lipschitz.

We denote the $L^2$-weighted spaces with weights $h$ by 
$L^2(h)$ and we equip them with the norm
\begin{align}
\norm{f}_{L^2(h)} := \int_{\MovingDomT} h |f|^2 \, dx.
\nonumber
\end{align}
With these notations the base $L^2$ space of pairs of functions in $\Omega$ for our system, denoted by $\H$, is defined as 
\[
\H = L^2(r^{\frac{1-\kappa}{\kappa}}) \times L^2(r^{\frac{1}{k}}).
\]
This space depends only on the choice of $r$. However, we 
will often use an equivalent norm that also depends on $v$, 
which corresponds to the energy space for the linearized problem
 and will also be important in the construction of our energies:
\begin{align}
\norm{(s,w)}^2_{\mathcal{H}} = \int_{\MovingDomT}     r^{\frac{1-\kappa}{\kappa}}  ( s^2  
+ a_2^{-1} r \GInvP^{ij} w_i w_j )\, dx.
\label{E:Basic_H_norm}
\end{align}
This uses $\GInv$ to measure the pointwise norm of the one form
$w$. The $\mathcal{H}$ norm is equivalent to 
the $\mathcal{H}^0$
norm 
(see the definition of $\mathcal{H}^{2k}$ below)
since $\GInv$ is equivalent to 
the the Euclidean inner product with constants depending on
the $L^\infty$ norm of $(r,v)$.

We continue with higher Sobolev norms.
We define $H^{j,\sigma}$, where $j\geq 0$ is an integer and
$\sigma > -\frac{1}{2}$, 
to be the space of all distributions in $\MovingDomT$
whose norm
\begin{align}
\norm{ f }^2_{H^{j,\sigma}} := \sum_{ \vert \alpha \vert \leq j} 
\norm{ r^\sigma \partial^{\alpha}f }_{L^2}^2
\nonumber
\end{align}
is finite. Using interpolation, we extend this definition, thus
defining $H^{s,\sigma}$ for all real $s \geq 0$.

To measure higher regularity we will also need higher Sobolev spaces where the weights depend on the number of derivatives.
More precisely, we define $\mathcal{H}^{2k}$ as the space of pairs of functions $(s,w)$ defined inside $\MovingDom_t$, and
for which the norm below is finite :
\begin{align}
\norm{(s,w)}^2_{\mathcal{H}^{2k}} := 
\sum_{\vert \alpha \vert =0}^{2k} 
\sum_{\substack{a=0 \\
|\alpha|- a\leq k}}^{k} \norm{ r^{\frac{1-\kappa}{2\kappa} + a} \partial^\alpha  s}^2_{L^2}
+
\sum_{\vert \alpha \vert =0}^{2k} 
\sum_{\substack{a=0 \\
|\alpha|- a\leq k}}^{k}
\norm{  r^{\frac{1-\kappa}{2\kappa}+\frac{1}{2} + a} \partial^\alpha w }^2_{L^2}.
\nonumber
\end{align}
We extend the definition of $\mathcal{H}^{2k}$ to non-integer
$k$ using interpolation. An explicit characterization 
of $\mathcal{H}^{2k}$ for non-integer $k$, based on interpolation, was given
in the last two authors' prior work \cite{IT-norel}.
Using the embedding theorems given below, we can
show that the $\mathcal{H}^{2k}$ norm is equivalent to the 
$H^{2k,\frac{1-\kappa}{2\kappa}+k} \times 
H^{2k,\frac{1}{2\kappa}+k}$ norm.

\subsection{The state space \texorpdfstring{$\bfH^{2k}$}{}.}

As already mentioned in the introduction, the state space $\bfH^{2k}$
is defined for $k > k_0$ (i.e. above scaling) as the set of 
pairs of functions $(r,v)$ defined in a domain $\Omega_t$ in $\R^d$ with boundary $\Gamma_t$ with the following properties:

\begin{enumerate}[label=\roman*)]
\item Boundary regularity: $\Gamma_t$ is a Lipschitz surface.

\item Nondegeneracy: $r$ is a Lipschitz function in $\bar \Omega_t$, positive
inside $\Omega_t$ and vanishing simply on the boundary $\Gamma_t$.

\item Regularity: The functions $(r,v)$ belong to $\H^{2k}$.

\end{enumerate}

Since the domain $\Omega_t$ itself depends on the function $r$, one cannot think of $\bfH^{2k}$ as a linear space, but rather as an infinite dimensional 
manifold. 
However, 
describing a manifold structure for $\bfH^{2k}$ is beyond the purposes
of our present paper, particularly since the trajectories associated with our flow are merely expected to be continuous with values in $\bfH^{2k}$.
For this reason, here we will limit ourselves to defining a topology on $\bfH^{2k}$.

\begin{definition}\label{d:convergence}
A sequence $(r_n,v_n)$ converges to $(r,v)$ in $\bfH^{2k}$ if the following 
conditions are satisfied:
\begin{enumerate}[label=\roman*)]
\item Uniform nondegeneracy, $|\nabla r_n| \geq c > 0$.

\item Domain convergence, $\|r_n - r\|_{Lip} \to 0$. Here, 
we consider the functions $r_n$ and and $r$ as extended
to zero outside their domains, giving rise to Liptschitz
functions in $\R^d$.

\item Norm convergence: for each $\epsilon > 0$ there exist 
a smooth function 
$(\tilde{r},\tilde{v})$
in a neighbourhood of $\Omega$ so that 
\[
\| (r,v)- (\tilde r,\tilde v)\|_{\H^{2k}(\Omega)} \leq \epsilon, \qquad 
\limsup_{n \to \infty}  \| (r_n, v_n)- (\tilde r,\tilde v)\|_{\H^{2k}(\Omega_n)} \leq \epsilon.
\]
\end{enumerate}
\end{definition}

The last condition in particular provides both a uniform bound for the sequence $(r_n,v_n)$ in $\H^{2k}(\Omega_n)$ as well as an equicontinuity type property, insuring that a nontrivial 
portion of their $\H^{2k}$ norms cannot concentrate on thinner layers
near the boundary. This is akin to the the conditions in the Kolmogorov-Riesz theorem for compact sets in $L^p$ spaces.

This definition will enable us to achieve two key properties of our flow:
\begin{itemize}
    \item Continuity of solutions $(r,v)$ as functions of $t$ with values
    in $\bfH^{2k}$.
    \item Continuous dependence of solutions $(r,v) \in \bfH^{2k}$ as functions of the initial data $(\mathring{r},\mathring{v}) \in \bfH^{2k}$.
\end{itemize}

\subsection{Regularization and good kernels}
In what follows we outline the main steps developed in Section~2 of \cite{IT-norel}, and in which, for a given state $(r,v)$ in $\bfH^{2k}$, we construct regularized states, denoted by $(r^h, v^h)$, to our free boundary evolution, associated to a dyadic frequency scale $2^h$, $h \geq 0$. This relies on having good regularization operators associated to each dyadic frequency scale $2^h$, $h \geq 0$. We denote these regularization operators by $\Psi^h$, with kernels $K^h$. These are the same  as in \cite{IT-norel}, and their exact definition can be found in there as well. A brief description on how one should envision these  regularization operators  is in order.

It is convenient to think of the  domain $\Omega_t$ as partitioned in dyadic boundary layers, denoting by $\Omega^{[j]}$ the layer at distance $2^{-2j}$ away from the boundary. Within each boundary layer we need to understand which is the correct spatial regularization scale. The principal part of the second order elliptic differential operator associated to our system is the starting point.  Given a dyadic frequency scale $h$, our regularizations will need to select frequencies $\xi$ with the property that $r \xi^2 \lesssim 2^{2h}$, which would require 
kernels on the dual scale 
\[
\delta x \approx r^{\frac12} 2^{-h}.
\]
However, if we are too close to the boundary, i.e. $r \ll 2^{-2h}$,
then we run into trouble with the uncertainty principle, as we would have
$\delta x \gg r$. To remedy this issue we select the spatial scale 
$r \lesssim 2^{-2h}$ and the associated frequency scale $2^{2h}$ as cutoffs in this analysis. Then the way the regularization works is as follows: (i) for $j < h$, the regularizations $(r^h,v^h)$ in $\Omega^{[j]}$ are determined by 
$(r,v)$ also in $\Omega^{[j]}$, and (ii) for $j=h$, the values of $(r,v)$ in $\Omega^{[h]}$ determine $(r^h, v^h)$ in a full neighborhood $\tilde{\Omega}^{[>h]}$ of $\Gamma$, of size $2^{-2h}$. The regularized state is obtained by restricting the full regularization to the domain $\Omega_h :=\left\{ r^h>0\right\}$.

\smallskip

For completeness we state the result in \cite{IT-norel}, and refer the reader  there for the proof:
\begin{proposition}\label{p:reg-state}
Assume that $k > k_0$. Then given a state $(r,v) \in \bfH^{2k}$, there 
exists  a family of regularizations $(r^{h},v^h) \in \bfH^{2k}$,
so that the following properties hold for 
a slowly varying frequency envelope $c_h \in \ell^2$ which satisfies
\begin{equation}\label{app-fe}
\|c_h\|_{\ell^2} \lesssim_A    \| (r,v) \|_{\bfH^{2k}}   .
\end{equation}

\begin{enumerate}[label=\roman*)]
\item  Good approximation, 
\begin{equation}\label{app-point}
   (r^{h},v^h) \to (r,v) \quad \text{  in }  C^1 \times C^\frac12 
   \quad \text{ as } h \to \infty,
\end{equation}
and 
\begin{equation}\label{app-point0}
   \|r^{h} - r\|_{L^\infty(\Omega)} \lesssim 2^{-2(k-k_0+1) h }.
\end{equation}

\item Uniform bound,
\begin{equation}\label{app-uniform}
\| (r^h,v^h) \|_{\bfH^{2k}}   \lesssim_{A} \| (r,v) \|_{\bfH^{2k}}  .
\end{equation}

\item  Higher regularity
\begin{equation}\label{app-high}
\| (r^h,v^h) \|_{\bfH^{2k+2j}_h}   \lesssim  2^{2hj}  c_h, \qquad j > 0.
\end{equation}

\item Low frequency difference bound:
\begin{equation}\label{app-diff}
\| (r^{h+1},v^{h+1}) - (r^h,v^h) \|_{\H_{\tilde r}}   \lesssim  2^{-2hk}  c_h,
\end{equation}
for any defining function $\tilde{r}$ with the property  $|\tilde r - r| \ll 2^{-2h}.$
\end{enumerate}
\end{proposition}

\subsection{Embedding and interpolation theorems\label{S:Embedding_and_interpolation}}
In this section we state some embedding and interpolation
results that will be used throughout.They have been
proved in the last two authors' prior paper \cite{IT-norel}, to which
the reader is referred to for the proofs.

\begin{lemma}\label{l:hardy}
Assume that $s_1 > s_2 \geq 0$ and  $\sigma_1 > \sigma_2 > -\frac12$ 
with $s_1-s_2 = \sigma_1-\sigma_2$.
Then we have 
\begin{equation}
H^{s_1,\sigma_1} \subset H^{s_2,\sigma_2}.
\nonumber 
\end{equation}
\end{lemma}

As a corollary of the above lemma we have embeddings into standard Sobolev spaces:
\begin{lemma}\label{l:sobolev}
Assume that $\sigma > 0$ and $\sigma \leq j$.
Then we have 
\begin{equation}
H^{j,\sigma} \subset H^{j-\sigma}.
\nonumber
\end{equation}
\end{lemma}

In particular, by standard Sobolev embeddings,  we also have Morrey type  embeddings into $C^s$ spaces:
\begin{lemma} \label{l:morrey}
We have 
\begin{equation}\
H^{j,\sigma} \subset C^{s}, \qquad   0 \leq s \leq  j- \sigma - \frac{d}2,
\nonumber
\end{equation}
where the equality can hold only if $ s$ is not an integer.
\end{lemma}

Next, we state the interpolation bounds:

\begin{proposition}\label{p:interpolation-g}
Let $\sigma_0, \sigma_m \in \mathbb{R}$ and $1\leq p_0,p_m\leq \infty$. Define 
\begin{equation}
\nonumber
\theta_j =   \frac{j}{m}, \qquad \frac{1}{p_j} = \frac{1-\theta_j}{p_0}+\frac{\theta_j}{p_m}, \qquad \sigma_j = 
\sigma_0(1-\theta_j) + \sigma _m\theta_j,
\end{equation}
and assume that
\begin{equation}
m - \sigma_m - d\left( \frac{1}{p_m}-\frac{1}{p_0} \right) > -\sigma_0, \qquad \sigma_j >-\frac{1}{p_j}.
\nonumber
\end{equation}
Then  for $0 < j < m$ we have
\begin{equation}
\nonumber
\| r^{\sigma_j } \partial^j f \|_{L^{p_j}} \lesssim \|  r^{\sigma_0}  f \|_{L^{p_0}}^{1-\theta_j} \|r^{\sigma_m}\partial^m f \|_{L^{p_m}}^{\theta_j} . 
\end{equation}
\end{proposition}

\begin{remark}
One particular case of the above proposition which will be used later is when  $p_0=p_1=p_2=2$, with the corresponding relation in between the exponents of the  $r^{\sigma_j}$ weights.
\end{remark} 

As the objective here is to interpolate
between the $L^2$ type $H^{m,\sigma}$ norm and  $L^\infty$ bounds, we will need the following straightforward consequence of Proposition~\ref{p:interpolation-g}:

\begin{proposition}\label{p:interpolation}
Let  $\sigma_m > -\frac12$ and 
\begin{equation}
m - \sigma_m - \frac{d}2 > 0.
\nonumber 
\end{equation}
Define
\begin{equation}
\nonumber
\theta_j =   \frac{j}{m}, \qquad \frac{1}{p_j} = \frac{\theta_j}2, \qquad \sigma_j =  \sigma_m \theta_j.
\end{equation}
Then  for $0 < j < m$ we have
\begin{equation}
\nonumber
\| r^{\sigma_j } \partial^j f \|_{L^{p_j}} \lesssim \|   f \|_{L^\infty}^{1-\theta_j} \|r^{\sigma_m}\partial^{m}f \|_{L^{2}}^{\theta_j}  .
\end{equation}
\end{proposition}

We will also need the following two variations of Proposition~\ref{p:interpolation}:

\begin{proposition}
\label{p:interpolation-c}
Let  $\sigma_m >-\frac{1}{2}$ and 
\[
m-\frac{1}{2}-\sigma_m -\frac{d}{2}>0.
\]
Define 
\[
\sigma_j=\sigma_m\theta_j,\quad \theta_j=\frac{2j-1}{2m-1}, \quad \frac{1}{p_j}=\frac{\theta_j}{2}.
\]
Then for $0<j<m$ we have
\[
\Vert r^{\sigma_j}\partial ^j f\Vert_{L^{p_j}}\lesssim \Vert f\Vert^{1-\theta_j}_{\dot{C}^{\frac{1}{2}}}  \Vert r^{\sigma_m}\partial^m f \Vert^{\theta_j}_{L^{2}}.
\]
\end{proposition}

\begin{proposition}
\label{p:interpolation-d}
Let  $\sigma_m >\frac{m-2}{2}$ and 
\[
m-\frac{1}{2}-\sigma_m -\frac{d}{2}>0.
\]
Define 
\[
\sigma_j=\sigma_m\theta_j-\frac12(1-\theta_j),\quad \theta_j=\frac{j}{m}, \quad \frac{1}{p_j}=\frac{\theta_j}{2}.
\]
Then for $0<j<m$ we have
\[
\Vert r^{\sigma_j}\partial ^j f\Vert_{L^{p_j}}\lesssim \Vert f\Vert^{1-\theta_j}_{\tilde{C}^\frac{1}{2}}  \Vert r^{\sigma_m}\partial^m f \Vert^{\theta_j}_{L^{2}}.
\]
\end{proposition}

\section{The linearized equation\label{S:Linearized}}
Consider a one-parameter family of solutions $(r_\tau, v_\tau)$ for the main system \eqref{E:r_v_eq} such that 
$\left.(r_\tau,v_\tau)\right|_{\tau=0} = (r,v)$, Then formally the 
functions $\left. \frac{d}{d \tau} (r_\tau,v_\tau)\right|_{\tau=0} = (s,w)$,
defined in the moving domain $\MovingDom_t$, will solve the corresponding linearized equation.
Precisely, a direct computation shows that, for $(s,w)$ in $\MovingDom_t$, the linearized 
equation can be written in the form
\begin{subequations}{\label{E:Equations_linearized}}
\begin{align}[left = \empheqlbrace\,]
& \Dt s + \frac{1}{\kappa} \GInvP^{ij} \partial_i r w_j + r  \GInvP^{ij} \partial_i w_j 
+ r a_1  v^i \partial_i s \,
 = f
\label{E:s_eq_linear} 
\\
& \Dt w_i + a_2  \partial_i s  = g_i,
\label{E:w_eq_linear}
\end{align}
\end{subequations}
where $f$ and $g_i$ represent perturbative terms of the form
\[
f = V_1 s + r W_1 w, \qquad g = V_2 s + W_2 w
\]
with potentials $V_{1,2}$ and $W_{1,2}$ which are linear in $\partial (r,v)$, with coefficients which are smooth functions of $r$ and $v$.

Importantly, we remark that for the above system we do not obtain or require any boundary conditions on the free boundary $\Gamma_t$. This is related to the fact that our one parameter family of solutions are not required to have the same domain, as it would be the case if one 
were working in Lagrangian coordinates.

For completeness, we also provide the explicit expressions for the potentials $V_{1,2}$ and $W_{1,2}$, though  this will not play any role in the sequel. We have 
\begin{equation*}
\begin{aligned}
V_{1} \,\, \,&=    \ \frac{v^j}{(v^0)^3}\br^{1+\frac{2}{\kappa}} \partial_j r 
        - \GInvP^{ij} \partial_i v_j 
        - r \frac{\partial \GInvP^{ij}}{\partial r}
        \partial_i v_j ,  
               \\
W_{2}^{l}\,\, &=  \ - \frac{\partial \GInvP^{ij}}{\partial v^l}
        \partial_i v_j - r a_3 \GInvP^{il} \partial_i r ,
 \\
V_{2,i}\, \, &= 
        -\frac{v^j}{(v^0)^3}\br^{1+\frac{2}{\kappa}} \partial_j v_i 
        + \frac{\partial a_2}{\partial r} \partial_i r ,
            \\
(W_{2})_i^l &=  \ - \frac{a_0}{\kappa \br}  \GInvP^{jl}
        \partial_j v_i   + \frac{\partial a_2}{\partial v^l} \partial_i r  , 
\end{aligned}
\end{equation*}
where $a_3$ is a smooth function of $(r,v)$, given by
\begin{align}
    \frac{a_0}{\kappa \br } - \frac{1}{\kappa} 
    = r a_3,\qquad  a_3 = - \frac{1}{\br} \left( \frac{1}{2} + \frac{|v|^2}{(v^0)^2} \right).
    \label{E:a_3_def_and_relation}
\end{align}
As for the other coefficients, the particular form of $a_3$ is
not relevant, but we wrote it here for completeness.

\subsection{Energy estimates and well-posedness}
We now consider the well-posedness of the linearized problem
\eqref{E:Equations_linearized} in the time dependent space 
$\H$. For the purpose of this analysis, we will view $\H$
as a Hilbert space whose squared norm plays the role of 
the energy functional for the linearized equation,
\begin{equation}
\label{E:the-inner-product-I-asked-for-1000-times}
    E_{lin}(s,w) := \norm{(s,w)}^2_{\mathcal{H}} = \int_{\MovingDomT}     r^{\frac{1-\kappa}{\kappa}}  ( s^2  
+ a_2^{-1} r \GInvP^{ij} w_i w_j )\, dx.
\end{equation}
We will use this space for both the linearized equation and its adjoint. 
Our main result here is as follows:

\begin{proposition}
Let $(r,v)$ be a solution to \eqref{E:r_v_eq}. Assume that both
$r$ and $v$ are Lipschitz continuous and
that $r$ vanishes simply on the free boundary.
Then, the linearized equations 
\eqref{E:Equations_linearized}
are well-posed in $\mathcal{H}$, and the following estimate holds for solutions
$(s,w)$ to \eqref{E:Equations_linearized}:
\begin{align}\label{ee-lin}
\left| \frac{d}{dt} \norm{(s,w)}^2_{\mathcal{H}}
\right| \lesssim B 
\norm{(s,w)}^2_{\mathcal{H}}.
\end{align}
\end{proposition}

\begin{proof}
We first remark that $(f,g)$ are indeed perturbative terms,
as they satisfy the estimate
\begin{align}
\norm{(f, g)}_{\mathcal{H} } \lesssim
B \norm{(s, w)}_{\mathcal{H} }.
\label{E:Perturbative_terms_linearized_eq_are_good}
\end{align}
This in term follows from a trivial pointwise bound
on the corresponding potentials,
\[
\| V_{1,2} \|_{L^\infty} + \| W_{1,2}\|_{L^\infty} \lesssim B.
\]

We multiply \eqref{E:s_eq_linear} by $     r^{\frac{1-\kappa}{\kappa}}s$ and contract \eqref{E:w_eq_linear}
with $a_2^{-1} r^{\frac{1}{\kappa}}\GInvP^{ij} w_j$ to find
\begin{align}
\begin{split}
&    \frac{1}{2} r^\frac{1-\kappa}{\kappa} \Dt s^2
    + \frac{1}{\kappa} r^\frac{1-\kappa}{\kappa} \G^{ij}
    \partial_i r w_j s 
    + r^\frac{1}{\kappa} \G^{ij} \partial_i w_j s
    +\frac{1}{2} r^\frac{1}{\kappa} a_1 v^i \partial_i s^2
    = f r^\frac{1-\kappa}{\kappa} s,
    \\
 &   \frac{1}{2}a_2^{-1} r^\frac{1}{\kappa} 
    \G^{ij} \Dt(w_i w_j) + r^\frac{1}{\kappa} 
    G^{ij}w_j \partial_i s
    = a_2^{-1} r^\frac{1}{\kappa} \G^{ij} g_i w_j.
\end{split}    
\nonumber 
\end{align}
Next, we add the two equations above, noting
that the second and third terms on the LHS of the first equation
combine with the second term on the LHS of the second equation
to produce
\begin{align}
    \begin{split}
        \frac{1}{\kappa} r^\frac{1-\kappa}{\kappa} \G^{ij}
        \partial_i r w_j s 
        + r^\frac{1}{\kappa} \G^{ij} \partial_i w_j s
        +
        r^\frac{1}{\kappa} 
        G^{ij} w_j \partial_i s 
        &=
         \partial_i(r^\frac{1}{\kappa} )\G^{ij}  w_j s 
        + r^\frac{1}{\kappa} \G^{ij} \partial_i w_j s
        +
        r^\frac{1}{\kappa} G^{ij} w_j \partial_i s 
        \\
        &=
        \G^{ij} \partial_i (r^\frac{1}{\kappa} w_j s)
    \end{split}
    \nonumber 
\end{align}
This yields
\begin{align*}
    \begin{split}
    &\frac{1}{2} r^\frac{1-\kappa}{\kappa} \Dt s^2
    +\frac{1}{2}a_2^{-1} r^\frac{1}{\kappa} 
    \G^{ij} \Dt(w_i w_j) 
    +\frac{1}{2} r^\frac{1}{\kappa} a_1 v^i \partial_i s^2
    +
    \G^{ij} \partial_i ( r^\frac{1}{\kappa} w_j s) =
    f r^\frac{1-\kappa}{\kappa} s +
    a_2^{-1} r^\frac{1}{\kappa} \G^{ij} g_i w_j.
      \end{split}
    \label{E:Energy_identity}
\end{align*}
We now integrate the above identity over $\MovingDomT$, using the formula \eqref{E:Derivative_moving_domain} to produce a time derivative of the energy.  For this, we need to write the terms on the left as perfect  derivatives or material derivatives. When we do so the zero order coefficients do not cause any harm. We only need to be careful with the terms where a derivative falls on $r^\frac{1-\kappa}{\kappa}$ because this
could potentially produce a term with the wrong weight (i.e.,
one less power of $r$). However, this does not occur because
we can solve for $\Dt r$ in \eqref{E:r_eq}:
\begin{align}
    \Dt r^\frac{1-\kappa}{\kappa} 
    = \frac{1-\kappa}{\kappa} r^{\frac{1}{\kappa} -2} 
    \Dt r =
    r^\frac{1-\kappa}{\kappa} 
    O( \partial(r,v) ).
    \nonumber 
\end{align}
Using the above observations, we obtain
\begin{align}
\begin{split}
\left|
\frac{d}{dt} \norm{(s,w)}^2_{\mathcal{H}}
\right|
\lesssim  B \norm{(s,w)}^2_{\mathcal{H}} +
\norm{(s,w)}_{\mathcal{H}} \norm{(f,g)}_{\mathcal{H}} \lesssim  B \norm{(s,w)}^2_{\mathcal{H}}.
\end{split}
\nonumber
\end{align}

We now compute the adjoint equation to \eqref{E:Equations_linearized} with respect to the duality relation defined by the $\H$ inner product determined by the norm  \eqref{E:the-inner-product-I-asked-for-1000-times}. The terms
$f$ and $g$ on the RHS of \eqref{E:Equations_linearized} are linear expressions in $s$ and $r w$ and
and in $s$ and $w$, respectively, with $\partial(r,v)$ coefficients. Thus, the source terms in the adjoint equation have the same structure as the original equation.
Let us write the LHS of \eqref{E:Equations_linearized} as
\begin{align}
\Dt \begin{pmatrix}
s\\
w
\end{pmatrix}
+ \mathsf{A}^i 
\partial_i
\begin{pmatrix}
s\\
w
\end{pmatrix}
 + \mathsf{B} 
\begin{pmatrix}
s\\
w
\end{pmatrix},
\nonumber 
\end{align}
where
\begin{align}
\mathsf{A}^i = 
\begin{bmatrix}
a_1 r v^i & r \GInvP^{ij} \\
a_2 \updelta^{il} & 0_{3\times 3}
\end{bmatrix}
\nonumber
\end{align}
and 
\begin{align}
\mathsf{B} = 
\begin{bmatrix}
0_{1\times 1} &  \frac{1}{\kappa} \GInvP^{ij}\partial_i r \\
0_{3\times 1} & 0_{3\times 3}
\end{bmatrix}.
\nonumber
\end{align}
With respect to the $\mathcal{H}$ inner product, the adjoint term corresponding to 
$\mathsf{A}^i \partial_i$ is 
\begin{align}
\tilde{\mathsf{A}}^i \partial_i = 
- \begin{bmatrix}
a_1 r v^i & r \GInvP^{ij} \\
a_2 \updelta^{il} & 0_{3\times 3}
\end{bmatrix}\partial_i 
- \begin{bmatrix}
0 & \frac{1}{\kappa}\GInvP^{ij}\partial_i r \\
\frac{1}{\kappa} r^{-1} a_2 \partial_l r & 0_{3\times 3}
\end{bmatrix}
\nonumber
\end{align}
modulo terms that are linear expressions in $\tilde{s}$ and $r \tilde{w}$ 
and in $\tilde{s}$ and $\tilde{w}$ (with $\partial(r,v)$ coefficients) in the first and second components,
respectively,
where $\tilde{s}$ and $\tilde{w}$ are elements of the dual. 
Similarly, the adjoint term corresponding to $\mathsf{B}$ is
\begin{align}
\tilde{\mathsf{B}} = 
\begin{bmatrix}
0_{1 \times 1} & 0_{1\times 3} \\
\frac{1}{\kappa} r^{-1} a_2 \partial_l r & 0_{3 \times 3}
\end{bmatrix}.
\nonumber
\end{align}
Combining these expressions, we see that the bad term on the lower left corner of the second matrix in  $\tilde{\mathsf{A}}^i \partial_i$ cancels with the corresponding
terms in $\tilde{B}$. Therefore, the adjoint problem is the same as the original one, modulo perturbative terms, and it therefore admits an energy estimate similar to the energy estimate 
\eqref{ee-lin} we have for the linearized equations \eqref{E:Equations_linearized}.

In a standard fashion, the forward energy estimate for the linearized equation and the 
backward in time energy estimate for the adjoint linearized equation yield uniqueness,
respectively existence of solutions for the linearized equation, as needed.
This guarantees the well-posedness of the linearized problem.
\end{proof}

\subsection{Second order transition operators\label{S:Transition_operators}} 
An alternative approach the linearized equations is to rewrite the linearized 
equations \eqref{E:Equations_linearized} as a second order system
which captures the wave-like part of the fluid associated with the propagation of sound. Applying $\Dt$ to \eqref{E:s_eq_linear} and using
\eqref{E:w_eq_linear} and vice-versa, and ignoring perturbative terms, we find
\begin{subequations}{\label{E:Transition_operators_not_symmetric}}
\noeqref{E:Transition_s_not_symmetric,E:Transition_w_not_symmetric}
\begin{align}
&\Dt^2 s  \approx \hat{L}_1 s,
\label{E:Transition_s_not_symmetric}
\\
&\Dt^2 w_i  \approx (\hat{L}_2 w)_i,
\label{E:Transition_w_not_symmetric}
\end{align} 
\end{subequations}
where
\begin{subequations}
\begin{align}
&\hat{L}_1 s  := 
r \partial_i \left(a_2 \GInvP^{ij} \partial_j s\right)
+ \frac{a_2}{\kappa} \GInvP^{ij} \partial_i r \partial_j
 s,
\label{E:L_1_def_not_symmetric} 
\\
&(\hat{L}_2 w)_i := a_2  \left( \partial_i (r \GInvP^{ml}  \partial_m  w_l) 
+ \frac{1}{\kappa} \GInvP^{ml} \partial_m r\partial_i w_l 
\right). 
\label{E:L_2_def_not_symmetric} 
\end{align}
\end{subequations}
Equations \eqref{E:Transition_operators_not_symmetric} are akin to wave equations in that the operators $\hat{L}_1$ and $\hat{L}_2$ satisfy elliptic estimates as proved in Section \ref{S:Energy_coercivity}.
More precisely, the operator $\hat{L}_2$ is associated with the divergence part of $w$, and it  satisfies elliptic estimates once it is combined with 
a matching curl operator $\hat L_3$. 

Even though in this paper we do not use the operators $\hat L_1$ and $\hat L_2$
directly in connection to the corresponding wave equation, they nevertheless 
play an important role at two points in our proof. Because slightly different
properties of $\hat L_1$ and $\hat L_2$ are needed at these two points, 
we will take advantage of the fact that only their principal part is 
uniquely determined in order to make slightly different choices for  $\hat L_1$ and $\hat L_2$.  Precisely, these operators will be needed as follows:

\begin{enumerate}[label=\Roman*.]
    \item In the proof of our energy estimates in Section~\ref{S:Energy_coercivity}, in order to establish the coercivity of our energy functionals.  There 
we will need the coercivity of $\hat L_1$ and $\hat L_2 + \hat L_3$, but 
we also want their coefficients to involve only $r,\nabla r$ and undifferentiated $v$.

 \item In the constructive proof of existence of regular solutions, in 
 our paradifferential style regularization procedure. There we use functional calculus for both $\hat L_1$ and $\hat L_2 + \hat L_3$, so we need them to 
be both coercive and self-adjoint, but we no longer need to impose the previous 
restrictions on the coefficients.
\end{enumerate}

The two sets of requirements are not exactly\footnote{Heuristically, both would 
be fulfilled by an appropriate Weyl type paradifferential quantization, but that would be very cumbersome to use in the presence of the free boundary.} compatible, which is why two choices are needed. 

We begin with the case (I), where we modify $\hat L_1$ and $\hat L_2$ as follows:
\begin{subequations}
\begin{align}
&\tilde {L}_1 s  := 
\GInvP^{ij} a_2 \left( r\partial_i \partial_j s 
+ \frac{1}{\kappa} \partial_i r \partial_j
 s \right),
\label{E:tL_1} 
\\
&(\tilde {L}_2 w)_i := a_2 \GInvP^{ml} \left( \partial_i (r   \partial_m  w_l) 
+ \frac{1}{\kappa} \partial_m r\partial_i w_l 
\right).
\label{E:tL_2}
\end{align}
\end{subequations}

 To $\tilde L_2$ we associate an operator $\tilde L_3$ of the form
\begin{align}
     ({\tilde L}_3 w)^i := r^{-\frac{1}{\kappa}} a_2 G^{ij} \partial^l [ r^{1+\frac{1}{\kappa}} (\partial_l w_j - \partial_j w_l) ].
\label{E:L_3_tilde_def}
\end{align}

For case (II), we keep the first of the operators, setting $L_1 = \hat L_1$  but make some lower order changes to $\hat L_2$ and $\hat L_3$ as follows:
\begin{align}
({L}_2 w)_i := \partial_i\left( a_2^2 (r \partial_m 
+ \frac{1}{\kappa} \partial_m r)(a_2^{-1} \GInvP^{ml}   w_l)\right) .
\label{E:L_2_def}
\end{align}
\begin{align}
     (L_3 w)_i := r^{-\frac{1}{\kappa}} a_2 F_{ij} \partial_l [
    G^{lm} G^{jp} r^{1+\frac{1}{\kappa}} (\partial_m w_p - \partial_p w_m) ],
\label{E:L_3_tilde_def+}
\end{align}
where $\F$ is the inverse of the matrix $\G$, i.e.,
$\F = \G^{-1}$.

It is not difficult to show that $L_1$ is a self-adjoint operator in $L^2(r^\frac{1-\kappa}{\kappa})$ with respect to the inner product defined by the first component of the $\H$ norm in \eqref{E:Basic_H_norm}, and 
\[
\mathcal{D}(L_1)  = \left\{ 
f \in L^2(r^{\frac{1-\kappa}{\kappa}}) \, | \, L_1 f \in L^2(r^{\frac{1-\kappa}{\kappa}}) \text{ in the sense of distributions} 
\right\}.
\]

Similarly, both $L_2$ and $L_3$ are  self-adjoint operators in $L^2(r^\frac{1}{\kappa})$  with respect to the inner product defined by the 
second
component of the $\H$ norm in \eqref{E:Basic_H_norm}
and 
\[
\mathcal{D}(L_2)  = \left\{ 
f \in L^2( r^{\frac{1}{\kappa}} ) \, | \, L_2 f \in L^2( r^{\frac{1}{\kappa}} ) \text{ in the sense of distributions} 
\right\}.
\]
and similarly for $L_3$.
We further note that $L_2L_3= L_3L_2=0$ 
and that the output of $L_2$ is a gradient, whereas
$L_3 w$ depends only on the curl of $w$.

As seen above, it is the operator
$\hat{L}_2$ that naturally come out of the equations of motion rather
than $L_2$ (recall that $L_1 = \hat{L}_1$). Thus, we need to compare these operators; we also compare $L_1$ and $\tilde{L}_1$ for later reference.
We have
\begin{equation}{\label{E:L_hat_no_hat_relation}}
\left\{
\begin{aligned}
&(L_2 w)_i = (\hat{L}_2 w)_i 
+ \partial_i a_2 r \partial_m (\G^{ml} w_l) 
+ \partial_i (r a_2^3 \partial_m a_2^{-1} \G^{ml} w_l)
+\frac{a_2}{\kappa} \partial_i (\G^{ml} \partial_m r) w_l
\\
&L_1 s = \tilde{L}_1 s + r \partial_i(a_2 \G^{ij}) \partial_j s.
\end{aligned}
\right.
\end{equation}

We will establish coercive estimates for $L_1$, $L_2$, and
$L_3$ (see Sections \ref{S:Energy_estimates}
and \ref{S:Regular_solutions}), from which follows
that the above domains can be characterized as
\begin{align}
    \begin{split}
    \mathcal{D}(L_1) = H^{2,\frac{1+\kappa}{2\kappa}},\qquad \mathcal{D}(L_2+L_3) = H^{2,\frac{1+3\kappa}{2\kappa}}.
    \end{split}
    \nonumber 
\end{align}

\section{The uniqueness theorem\label{S:Uniqueness}}
In this Section we establish Theorem 
\ref{T:Uniqueness}. It will be a direct consequence
of the more general Theorem \ref{t:Diff} below,
which establishes that a suitable functional
that measures the difference between two solutions
is propagated by the flow.

We consider two solutions $(r_1,v_1)$ and
$(r_2,v_2)$ defined in the respective
domains $\Omega^1_t$ and $\Omega^2_t$.
Put $\MovingDomT := \Omega_t^1 \cap \Omega_t^2$, $\Gamma_t:= \partial \MovingDomT$. If the boundaries
of the domains $\Omega^1_t$ and $\Omega^2_t$ are
sufficiently close, which will be the case of 
interest here, then $\MovingDomT$ will
have a Lipschitz boundary.
Let $\Dt^1$ and $\Dt^2$ be the material derivatives associated with the domains $\Omega_t^1$ and
$\Omega_t^2$, respectively. In $\Omega_t$ define the 
averaged material derivative
\begin{align}
    \Dt := \frac{\Dt^1+\Dt^2}{2}
    \nonumber
\end{align}
and the averaged $\GInv$,
\[
\GInvP^{ij}_{mid} := \GInvP^{ij} \left(\frac{r_1+r_2}{2}, \frac{v_1+v_2}{2}\right).
\]
We note that the above averaged material derivative is not exactly advecting the domain $\Omega_t$. Fortunately exact advection is not at all needed for what follows. See also Remark~\ref{r:no-advection} below.

To  measure the difference between  two solutions on the common domain $\Omega_t$, we introduce the following distance functional:
\begin{equation}
    \label{distance functional}
    \D_{\H}((r_1,v_1), (r_2,v_2)) :=
    \int_{\Omega_t}
    (r_1+r_2)^\frac{1-\kappa}{\kappa} (
    (r_1-r_2)^2 +(r_1+r_2) |v_1-v_2|^2)\, dx
\end{equation}
which is the same as in \cite{IT-norel}.
We could have used $\GInv$ to measure
$|v_1-v_2|$, but the Euclidean metric suffices. 
This is not only because both metrics are comparable
but also because we will not use \eqref{distance functional}
directly in conjunction with the equations. We will,
however, use $\GInv$ further below when we
introduce another functional for which the structure
of the equations will be relevant.

We observe the following Lemma, which has
been proved in \cite{IT-norel}.
\begin{lemma}
Assume that $r_1$ and $r_2$ are  uniformly Lipschitz and nondegenerate, and close
in the Lipschitz topology. Then we have 
\begin{equation}\label{Diff-bdr}
 \int_{\Gamma_t} |r_1+r_2|^{\frac{1}{\kappa} +2} d\sigma \lesssim    \D_\H((r_1,v_1),(r_2,v_2)) .
\end{equation}
\end{lemma}
One can view the integral in \eqref{Diff-bdr} as a measure of the distance between 
the two boundaries, with the same scaling as $\D_\H$. 

We now state our main estimate for differences of solutions:

\begin{theorem}\label{t:Diff}
Let $(r_1,v_1)$ and $(r_2,v_2)$ be two solutions for the system \eqref{E:r_v_eq} in $[0,T]$, with regularity $\nabla r_j \in \tilde{C}^\frac{1}{2}$, $v_j \in C^1$, so that $r_j$ are uniformly nondegenerate near the boundary and close in the Lipschitz topology, $j=1,2$. Then
we have the uniform difference bound
\begin{equation}\label{Diff-est}
\sup_{t \in [0,T]} \D_\H((r_1,v_1)(t),(r_2,v_2)(t))    \lesssim \D_\H((r_1,v_1)(0),(r_2,v_2)(0)). \end{equation}
\end{theorem}
We remark that 
\[
\D_\H((r_1,v_1),(r_2,v_2)) = 0 \quad \text{iff} \quad (r_1,v_1)= (r_2,v_2),
\]
which implies our uniqueness result.

The remaining of this Section is dedicated to proving
Theorem \ref{t:Diff}.

\subsection{A degenerate energy functional}
We will not work directly with the functional
$D_\H$ because it is non-degenerate, so 
we cannot take full advantage of integration by parts
when we compute its time derivative. We thus 
consider the modified difference functional
\begin{align}
\label{Diff}
\begin{split}
 \tD_\H((r_1,v_1),(r_2,v_2)) & := \int_{\Omega_t} 
 (r_1+r_2)^{\frac{1-\kappa}{\kappa}}
\big(
 a (r_1-r_2)^2 
 \\
 &\ \  \qquad \  + b(a_{21}+a_{22})^{-1} \GInvP^{ij}_{mid} (v_1-v_2)_i (v_1-v_2)_j \big)\, dx,
 \end{split}   
\end{align}
where $a_{21}$ and $a_{22}$ are the 
coefficient $a_2$ corresponding to the solutions 
$(r_1,v_1)$ and $(r_2,v_2)$, respectively, and
$a$ and $b$ are functions of
$\mu := r_1+r_2$ and $\nu=r_1-r_2$ with the
following properties
\begin{enumerate}
\item They are smooth, nonnegative  functions in the region $\{ 0 \leq |\nu| < \mu\}$,
even in $\nu$, and homogeneous of degree $0$, respectively $1$.

\item They are connected by the relation $\mu a =  b$.

\item They are supported in $\{ |\nu| < \frac12 \mu\}$, with $a = 1$ in $\{|\nu| < \frac14 \mu \}$.
\end{enumerate}

\begin{remark}\label{r:no-advection}
The choice of the weights $a$ and $b$ above guarantees that the integrand in \eqref{Diff}
above vanishes polynomially on the boundary of the common domain $\Omega_t$.
This is why we refer to this difference functional as degenerate, and
is also the reason why we are able to use the averaged material derivative $D_t$ to propagate bounds for $\tD_\H$ in time even though  it does not exactly advect $\Omega_t$. 
\end{remark}

From \cite{IT-norel} we also borrow the equivalence property of the two distance functionals defined above:
\begin{lemma}\label{l:D-equiv}
Assume that $A = A_1+A_2$ is small. Then
\begin{equation}\label{D-equiv}
  \D_\H((r_1,v_1),(r_2,v_2)) \approx_A  \tD_\H((r_1,v_1),(r_2,v_2)) .
\end{equation}
\end{lemma}

\medskip

\subsection{The energy estimate}
To prove the Theorem~\ref{t:Diff} it remains to track the time evolution of the degenerate distance functional $\tD_{\H}$. This is the content of
the next result, which  
immediately implies Theorem~\ref{t:Diff}
after an application of 
Gronwall's inequality.

\begin{proposition}
We have
\begin{equation}\label{Diff-est0}
\frac{d}{dt} \tD_\H  ((r_1,v_1),(r_2,v_2)) \lesssim (B_1+B_2) \D_\H  ((r_1,v_1),(r_2,v_2)).
\end{equation}
\end{proposition}

\begin{proof}

The difference of the two solutions to \eqref{E:r_v_eq} in the common domain $\MovingDomT$ satisfies
\begin{equation}
\label{eq1}
    \begin{aligned}
 2 \Dt (r_1 - r_2) = & 
- (\Dt^1 - \Dt^2) (r_1 + r_2) 
- (r_1 (\GInv_1)^{ij} + r_2 (\GInv_2)^{ij} )\partial_i ( (v_1)_j - (v_2)_j )
\\
& 
- (r_1 (\GInv_1)^{ij} - r_2 (\GInv_2)^{ij} )\partial_i ( (v_1)_j + (v_2)_j )
\\
&
-
(r_1 a_{11} v_1^i + r_2 a_{12} v_2^i )\partial_i (r_1 - r_2)
\\
& 
-
(r_1 a_{11} v_1^i - r_2 a_{12} v_2^i )\partial_i (r_1 + r_2) ,
\end{aligned}
\end{equation}
and
\begin{equation}
\label{eq2}
    \begin{aligned}
 2\Dt( (v_1)_i - (v_2)_i ) = &  
- (\Dt^1 - \Dt^2 )( (v_1)_i + (v_2)_i )
- (a_{2,1} + a_{2,2}) \partial_i (r_1 - r_2)
\\
&
- (a_{21} - a_{22}) \partial_i (r_1 + r_2).
\end{aligned}
\end{equation}
Above, $\GInv_i$ and $a_{1i}$ correspond
to $\GInv$ and $a_1$ for the solutions
$(r_i,v_i)$, $i=1,2$.
We observe that the difference of material 
derivatives can be written as 
\[
(\Dt^1 - \Dt^2 ) = \left(\tilde v_1 - \tilde v_2 \right) \cdot \nabla,
\qquad \tilde v^i = \frac{v^i}{v^0}.
\]

To compute the time derivative of the degenerate distance
we use the standard formula for differentiation in a moving domain $\Omega_t$,
\begin{equation}
    \frac{d}{dt}\int_{\Omega_t} f(t, x)\,dx= \int_{\Omega_t} D_tf + f \nabla \cdot v (t)\, dx,
\end{equation}
where $v$ is in our case the average velocity. Classically this holds under the assumption that the domain $\Omega_t$ is advected by $D_t$. But in our case we replace this assumption
with the alternative condition that $f$ vanishes on the boundary of $\Omega_t$. Using this formula we obtain
\begin{equation}
\begin{aligned}
    \frac{d}{dt}\tD_{\H}(t)&=I_1+I_2+I_3+I_4+I_5
    +I_6+O(B_1+B_2)D_{\H}(t),
    \end{aligned}
    \nonumber 
\end{equation}
where the integrals $I_i$, $i=\overline{1,6}$, represent contributions as follows:
\medskip

a) $I_1$ represents the contribution where the averaged material derivative  falls on $a$ or $b$,
\begin{equation}
\begin{aligned}
    I_1= & \int_{\Omega_t} 
    (r_1+r_2)^\frac{1-\kappa}{\kappa} [a_\mu (r_1-r_2)^2 + b_\mu (a_{21}+a_{22})^{-1} \GInvP^{ij}_{mid} (v_1-v_2)^2\,] D_t(r_1+r_2) \, dx
    \\
   &  + \int_{\Omega_t}(r_1+r_2)^\frac{1-\kappa}{\kappa} [a_\nu (r_1-r_2)^2 + b_\nu (a_{21}+a_{22})^{-1} \GInvP^{ij}_{mid} (v_1-v_2)^2\,] D_t(r_1-r_2) \, dx.
    \end{aligned}
    \nonumber 
\end{equation}
Here the derivatives of $a$ and $b$ are homogeneous of order $-1$, respectively $0$. We get Gronwall terms when they  get coupled with factors of $r_1+r_2$ or $r_1-r_2$ from the material derivatives. We discuss 
$I_1$ later.

\medskip

b) $I_2$ gathers the contributions where the averaged material derivative is applied to  $(a_{21}+a_{22})^{-1}$ and to $\GInvP^{ij}_{mid}$. These expressions are smooth functions of $(r_1, v_1), (r_2,v_2)$, and thus their derivatives are bounded by $B_1+B_2$, 
\begin{equation}
\begin{aligned}
    I_2=O(B_1+B_2)D_{\H}(t).
    \end{aligned}
    \nonumber
\end{equation}

c) $I_3$ represents the main contribution  of the averaged material derivative that falls onto $(r_1-r_2)$ respectively on $v_1-v_2$ which consists of the first and second terms 
on the RHS
in \eqref{eq1}, and the second term in \eqref{eq2}:
\begin{equation}
\begin{aligned}
    I_3=&-\int_{\Omega_t} 
    (r_1+r_2)^\frac{1-\kappa}{\kappa}
    a(r_1-r_2)\big[ (\tilde{v}^i_1 - \tilde{v}^i_2) \partial_i (r_1 + r_2) 
    \\
    & \ \ \  
+ (r_1 (\GInv_1)^{ij} + r_2 (\GInv_2)^{ij} )\partial_i ( (v_1)_j - (v_2)_j )\big]\, dx\\
&- \int_{\Omega_t} (r_1+r_2)^\frac{1-\kappa}{\kappa} 
b \GInvP^{ij}_{mid} ((v_1)_j-(v_2)_j) \partial_i (r_1 - r_2)\, dx .
    \end{aligned}
    \nonumber 
\end{equation}
This term will need further discussion.

\medskip

d)  In $I_4$ we place the contribution of the forth term on the RHS of \eqref{eq1}:
\begin{equation}
\begin{aligned}
    I_4= -\int_{\Omega_t} 
    (r_1+r_2)^\frac{1-\kappa}{\kappa}
    a(r_1-r_2)(r_1 a_{11} v_1^i + r_2 a_{12} v_2^i )\partial_i (r_1 - r_2)\, dx.
    \end{aligned}
    \nonumber
\end{equation}
This term will be discussed later.

\medskip
e) $I_5$ is given by the third and fifth terms 
on the RHS in \eqref{eq1} and the third term on the RHS from \eqref{eq2}
\begin{equation}
\begin{aligned}
 I_5 =
&
\int_{\Omega_t}
(r_1+r_2)^\frac{1-\kappa}{\kappa} a(r_1-r_2)  
(r_1 (\GInv_1)^{ij} - r_2 (\GInv_2)^{ij} )\partial_i ( (v_1)_j + (v_2)_j ) \, dx
\\
&-\int_{\Omega_t}
(r_1+r_2)^\frac{1-\kappa}{\kappa}
a(r_1-r_2)
(r_1 a_{11} v_1^i - r_2 a_{12} v_2^i )\partial_i (r_1 + r_2)  \, dx\\
   & -\int_{\Omega_t}  (r_1+r_2)^\frac{1-\kappa}{\kappa} b
   (a_{21}+a_{22})^{-1}\GInvP^{ij}_{mid} ((v_1)_j-(v_2)_j)(a_{21} - a_{22}) \partial_i (r_1 + r_2)\, dx.
    \end{aligned}
    \nonumber
\end{equation}
All of the terms in $I_5$ are straightforward  Gronwall terms. 

\medskip

f) $I_6$ contains the terms where
$\Dt$ falls on $(r_1+r_2)^\frac{1-\kappa}{\kappa}$:
\begin{equation*}
    \begin{aligned}
    I_6 = 
     \frac{1-\kappa}{\kappa} \! \int_{\Omega_t} \! 
 (r_1+r_2)^{\frac{1}{\kappa}-2}
\left( a (r_1-r_2)^2\! +\!  b(a_{21}\!+\!a_{22})^{-1} \GInvP^{ij}_{mid} (v_1-v_2)_i (v_1-v_2)_j \right) \Dt(r_1+r_2)\, dx.
    \end{aligned}
\end{equation*}
We will analyze $I_6$ later.

\bigskip

It remains to take a closer look at the integrals $I_1$, $I_3$, $I_4$, and $I_6$. We consider them
in succession.

\medskip

\emph{The bound for $I_1$.}
Here we write 
\[
2 D_t (r_1+r_2) = 2 D_t^1 r_1 + D_t^2 r_2 - (\tilde v_1 - \tilde v_2) \cdot \nabla(r_1-r_2), 
\]
and
\[
2 D_t (r_1-r_2) = 2 D_t^1 r_1 - D_t^2 r_2 - (\tilde v_1 - \tilde v_2)\cdot \nabla(r_1+ r_2) .
\]
The first two terms have size $O(B(r_1+r_2))$ so their contribution is a Gronwall term.
We are left with the contribution of the last terms, which  yields the expressions
\begin{equation}
\begin{aligned}
    I_1^a= & \int_{\Omega_t} 
    (r_1+r_2)^\frac{1-\kappa}{\kappa}
     a_\mu (r_1-r_2)^2  (\tilde v_1 - \tilde v_2)\cdot \nabla(r_1-r_2) \,  dx
    \\
    & +\!
    \int_{\Omega_t} \!
    (r_1+r_2)^\frac{1-\kappa}{\kappa}
    b_\mu (a_{21}+a_{22})^{-1} \GInvP^{ij}_{mid} 
    ((v_1)_i-(v_2)_i)((v_1)_j-(v_2)_j) 
     (\tilde v_1 - \tilde v_2)\!\cdot\! \nabla(r_1-r_2) \,  dx,
    \\
   I_1^b = &  \int_{\Omega_t} (r_1+r_2)^\frac{1-\kappa}{\kappa} 
   a_\nu (r_1-r_2)^2  (\tilde v_1 - \tilde v_2)\!\cdot\! \nabla(r_1+ r_2)\,  dx 
   \\
   & 
   + \!
   \int_{\Omega_t} \! (r_1+r_2)^\frac{1-\kappa}{\kappa}
   b_\nu (a_{21}+a_{22})^{-1} \GInvP^{ij}_{mid} 
   ((v_1)_i-(v_2)_i)((v_1)_j-(v_2)_j) 
   (\tilde v_1 - \tilde v_2)\!\cdot\! \nabla(r_1+ r_2)\,  dx .  \end{aligned}
   \nonumber
\end{equation}
For the second integral in both expressions, 
we bound $|\tilde v_1 - \tilde v_2| \lesssim |r_1-r_2|+|v_1-v_2|$
and estimate their part by 
\[
\int_{\Omega_t}
(r_1+r_2)^\frac{1-\kappa}{\kappa}
|r_1-r_2||v_1-v_2|^2 \, dx \lesssim D_\H(t)
\]
and 
\[
J_2 = \int_{\Omega_t}  (r_1+r_2)^\frac{1-\kappa}{\kappa}
|v_1-v_2|^3 \, dx,
\]
which is discussed later.

We are left with the first integrals in $I_1^a$ 
and $I_1^b$, which we record as 
\begin{align}
    J_1^a= & \int_{\Omega_t} 
    (r_1+r_2)^\frac{1-\kappa}{\kappa}
     a_\mu (r_1-r_2)^2  (\tilde v_1 - \tilde v_2)\cdot \nabla(r_1-r_2) \,  dx
     \nonumber
\end{align}
and
\begin{align}
       J_1^b = &  \int_{\Omega_t} (r_1+r_2)^\frac{1-\kappa}{\kappa} 
   a_\nu (r_1-r_2)^2  (\tilde v_1 - \tilde v_2)\cdot \nabla(r_1+ r_2)\,  dx.
   \nonumber
\end{align}
These integrals are also discussed later.

\bigskip

\emph{The bound for $I_3$.} 
For $I_3$,
we seek to capture the same cancellation that it is seen in the analysis of the linearized equation.
We look at the last term in $I_3$, use 
$b=a(r_1+r_2)$, and integrate by parts; if the derivatives falls on $\GInvP$ then this is a straightforward Gronwall term. We are left with three contributions, two of which we pair with the first two terms in $I_3$. We obtain 
\[
\begin{aligned}
 I_3=
& 
-\int_{\Omega_t} 
(r_1+r_2)^\frac{1-\kappa}{\kappa}
a(r_1-r_2)\left[ (\tilde{v}^i_1 - \tilde{v}^i_2) -\frac{1}{\kappa} \GInvP^{ij}_{mid} ((v_{1})_j-
(v_{2})_j )\right] \partial_i (r_1 + r_2) \, dx \\
& -\int_{\Omega_t} 
(r_1+r_2)^\frac{1-\kappa}{\kappa}
a(r_1-r_2) \left[(r_1 (\GInv_1)^{ij} + r_2 (\GInv_2)^{ij} ) - (r_1+r_2)\GInvP^{ij}_{mid}\right]\partial_i ( (v_1)_j - (v_2)_j )\, dx\\
&+\int_{\Omega_t}  
(r_1+r_2)^\frac{1}{\kappa}
\partial_i a \GInvP^{ij}_{mid} ((v_1)_j-(v_2)_j)  (r_1 - r_2)\, dx
+O(B_1+B_2) D_{\H}
\\
& = I_3^1+I_3^2+ J_1^c +O(B_1+B_2) D_{\H} .
\end{aligned}
\]

For the first integral, $I_3^1$ we expand the difference $\tilde{v}^i_1 - \tilde{v}^i_2$, seen as a function of $v_1$ and $v_2$, in a Taylor series around the center $(v_1+v_2)/2$. We have 
\[
\dfrac{\partial \tilde{v}_i}{\partial v_j} = \frac{1}{v^0} \left( \delta^{ij}-\frac{v_i v_j}{(v^0)^2}\right),
\]
where we recognize the matrix on the right as being the main part in $\GInv$. Thus, we can write 
\[
\left| (\tilde{v}^i_1 - \tilde{v}^i_2) -
\frac{1}{\kappa} \GInvP^{ij}_{mid} ((v_{1})_j-(v_{2})_j)
\right|\lesssim \vert r_1-r_2\vert + (r_1+r_2) \vert v_1-v_2\vert + \vert v_1-v_2\vert^3,
\]
where the quadratic $v_1-v_2$ terms cancel because we are expanding around the middle, and we used
\eqref{E:a_3_def_and_relation} to get the second term on the right.
The contributions of all of  the  terms in the last expansion are Gronwall terms.

For the second integral $I_3^2$ we have a simpler expansion 
\[
\left| (r_1 (\GInv_1)^{ij} + r_2 (\GInv_2)^{ij} ) - (r_1+r_2)\GInvP^{ij}_{mid}\right| \lesssim \vert r_1-r_2\vert + (r_1+r_2)\vert (v_{1})_j-(v_{2})_j\vert ^2,
\]
where all contributions qualify again as Gronwall terms.

Finally, the last integral, $J_1^c$, is estimated below.

\bigskip

\emph{The bound for $I_4$.} 
After an integration by parts we have
\begin{equation}
\begin{aligned}
    I_4 & = \frac{1}{2}\int_{\Omega_t} 
    (r_1+r_2)^\frac{1-\kappa}{\kappa}
    \partial_i a \, (r_1 a_{11} v_1^i + r_2 a_{12} v_2^i )(r_1 - r_2)^2\, dx
    \\
    &   + 
    \frac{1-\kappa}{2\kappa} \int_{\Omega_t} 
    (r_1+r_2)^{\frac{1}{\kappa} - 2}
    \partial_i (r_1+r_2)
    a(r_1 a_{11} v_1^i + r_2 a_{12} v_2^i )(r_1 - r_2)^2\, dx
    \\
    & + O(B_1+B_2) D_{\H}.
    \end{aligned}
    \nonumber
\end{equation}
Writing
\begin{align}
    \begin{split}
        |r_1 a_{11} v_1^i + r_2 a_{12} v_2^i| & \lesssim r_1+r_2,
    \end{split}
    \nonumber
\end{align}
both integrals are bounded by $O(B_1+B_2) D_{\H}$.

\bigskip

\emph{The bound for $I_6$.} 
We use $\Dt  (r_1+r_2) = 
D_t^1 r_1 + D_t^2 r_2 - (\tilde{v}_1 - \tilde{v}_2)
\cdot \nabla( r_1-r_2)$ where the first two terms
are bounded by $(B_1+B_2)(r_1+r_2)$ and yield Gronwall contributions.
Then we write
\begin{align}
    I_6 \lesssim J_1^d + J_2 + O(B_1+B_2) D_\H, 
    \nonumber
\end{align}
where 
\begin{equation}
    \begin{aligned}
    J_1^d &= 
    \frac{1-\kappa}{\kappa} \int_{\Omega_t} 
 (r_1+r_2)^{\frac{1}{\kappa}-2}
 a (r_1-r_2)^2 (\tilde{v}_1 - \tilde{v}_2)
\cdot \nabla( r_1-r_2)\, dx.
 \nonumber
    \end{aligned}
\end{equation}
\bigskip

To summarize the outcome of our analysis so far, we have proved that
\[
 \frac{d}{dt}\tilde{D}_{\H}(t) \leq   J_1^a + J_1^b + J_1^c + J_1^d + O(J_2) + O(B_1+B_2) D_\H.
\]

It remains to estimate $J_2$, $J_1^a$, $J_1^b$,
$J_1^c$,  and $J_1^d$, which we write here again for convenience:
\begin{equation}
    \begin{aligned}
    J_2 &= \int_{\Omega_t}  (r_1+r_2)^\frac{1-\kappa}{\kappa}
|v_1-v_2|^3 \, dx,
    \nonumber
    \\
    J_1^a &= \int_{\Omega_t} 
    (r_1+r_2)^\frac{1-\kappa}{\kappa}
     a_\mu (r_1-r_2)^2  (\tilde v_1 - \tilde v_2)\cdot \nabla(r_1-r_2) \,  dx,
    \nonumber
    \\
    J_1^b &=  \int_{\Omega_t} (r_1+r_2)^\frac{1-\kappa}{\kappa} 
   a_\nu (r_1-r_2)^2  (\tilde v_1 - \tilde v_2)\cdot \nabla(r_1+ r_2)\,  dx,
    \nonumber
    \\
    J_1^c &= \int_{\Omega_t}  
(r_1+r_2)^\frac{1}{\kappa}
\partial_i a \GInvP^{ij}_{mid} ((v_1)_j-(v_2)_j)  (r_1 - r_2)\, dx,
    \nonumber
    \\
    J_1^d &= 
    \frac{1-\kappa}{\kappa} \int_{\Omega_t} 
 (r_1+r_2)^{\frac{1}{\kappa}-2}
 a (r_1-r_2)^2 (\tilde{v}_1 - \tilde{v}_2)
\cdot \nabla( r_1-r_2)\, dx.
    \end{aligned}
\end{equation}

The integral $J_2$ is the same as in \cite{IT-norel} and can be estimated accordingly, 
using interpolation inequalities; see Lemma~4.4 in \cite{IT-norel}.

The bound for the integrals $J_1^a$, $J_1^b$, $J_1^c$ and $J_1^d$ matches 
estimates for similar integrals in \cite{IT-norel}. 
Precisely, the integrals  $J_1^a$ and $J_2^b$ are estimated as the integrals 
called $J_1^b$ and $J_1^c$ in \cite{IT-norel},
respectively, see Lemmas 4.6 to 4.13 in \cite{IT-norel}.
The integral $J_1^c$ is estimated as the integral
$J_1^d$ in \cite{IT-norel}, see Lemmas 4.6 to 4.13 in \cite{IT-norel}. The integral $J_1^d$ is estimated as the integral $J_1^a$ in \cite{IT-norel}, see Lemmas 4.6 to 4.13 in \cite{IT-norel}.

We caution the reader that the arguments in 
\cite{IT-norel} are not straightforward, and involve peeling off a carefully chosen boundary layer, with separate estimates inside the boundary layer and outside it.
The only difference in the present paper is the presence of additional weights in the 
integrals, which are smooth functions of $r_1,r_2,v_1,v_2$. For instance, the difference $\tilde v_1 - \tilde v_2$ can be expanded as
\[
\tilde v_1 - \tilde v_2 = f(r,v) (r_1-r_2) + g(r,v) (v_1-v_2),
\]
where $r,v$ stand for $r_1,r_2,v_1,v_2$ and $f,g$ are smooth. The contribution of the first term admits a straightforward Gronwall bound, and the contribution of the second term is akin to the corresponding integral in \cite{IT-norel} but with the added smooth weight. 
The point is that every time we integrate by parts and the derivative falls on the smooth weight, the corresponding contribution is a straightforward Gronwall term. Hence such smooth weights make no difference if added in the arguments in \cite{IT-norel}.
\end{proof}

\section{Energy estimates for solutions\label{S:Energy_estimates}}

Our goal in this section is to establish
uniform control over the $\mathbf{H}^{2k}$
norm of the solutions $(r,v)$ to \eqref{E:r_v_eq} with growth given
by the norms $A$ and $B$. For this, we will
use appropriate energy functionals $E^{2k}= E^{2k}(r,v)$ constructed
out of vector fields naturally associated
with the evolution. Our functionals will
be associated with the wave and transport
parts of the system, which will be considered
at matched regularity.

The vector fields we will consider are:
\begin{itemize}
    \item The material derivative $\Dt$, which has
    order $1/2$.
    \item All regular derivatives $\partial$, 
    which have order $1$.
    \item Multiplication by $r$, which has order $-1$.
\end{itemize}
The wave part of the energy will be associated mainly with $\Dt$, whereas the transport
part will be associated with all of the above
vector fields.

\subsection{Good variables and the energy functional\label{S:Good_variables_energy_functional}}
Heuristically, higher order energy functionals should be obtained by applying an appropriate 
number of vector fields to the equation, and then verifying that the output solves the linearized equation modulo perturbative terms. In the absence of the free boundary, there are two equally good
choices, (i) to spatially differentiate the equation, using the $\partial_j$ vector fields,
or (ii) to differentiate the equation in time, using the $\partial_t$ vector field.

However, in the free boundary setting, both of the above choices have issues, as neither 
$\partial_j$ nor $\partial_t$ are adapted to the boundary. For $\partial_t$ we do have a seemingly
better choice, namely to replace it by the material derivative $D_t$. However, this 
has the downside that it does not arise from a symmetry of the equations, and consequently 
the expressions $(D_t^{2k} r, D_t^{2k} v)$ are not good approximate solutions to the linearized equations. To address this matter, we add suitable corrections to  these expressions, obtaining  
what we call the \emph{good variables} $(s_{2k},w_{2k})$.
Precisely, motivated by the linearized equations \eqref{E:Equations_linearized}, we introduce
\begin{align}
\begin{split}
s_0 &:= r,
\\
w_0 & := v,
\\
s_1 & := \partial_t r,
\\
w_1 & := \partial_t v,
\\ 
s_2 & :=\Dt^2r + \frac{1}{2}
\frac{a_0 a_2}{\kappa \br} \GInvP^{ij} \partial_i r 
\partial_j r,
\\
(w_k)_i &:= \Dt^k v_i, \, \qquad  k \geq 2,
\\
s_k & := \Dt^k r -  \frac{a_0}{\kappa \br} \GInvP^{ij} \partial_i r \Dt^{k-1} v_j, \qquad \, k\geq 3.
\end{split}
\label{E:Good_variables_def}
\end{align}
Here, we use the full equations \eqref{E:r_v_eq} to 
interpret $(s_j,w_j)$ as multilinear expressions in $(r,v)$, with coefficients which are functions of undifferentiated $(r,v)$. Observe that it would be compatible
with the linearized equations to define $s_k$ with 
$\dfrac{1}{\kappa} \GInvP^{ij} \partial_i r \Dt^{k-1} v_j$ 
instead of 
$\dfrac{a_0}{\kappa \br} \GInvP^{ij} \partial_i r \Dt^{k-1} v_j$.
The difference between the two cases, however, is a perturbative
term due to \eqref{E:a_3_def_and_relation}, and the definition
we adopt here is more convenient because $\dfrac{a_0}{\kappa \br}$
is what appears in the commutator $[\Dt,\partial]$.

Using equations \eqref{E:r_v_eq}, we find that for $k \geq 1$, our good variables
$(s_{2k}, w_{2k})$ can be seen as 
approximate solutions to the linearized equation \eqref{E:Equations_linearized}. Precisely, they satisfy the following equations in $\MovingDom$ 
(compare with \eqref{E:Equations_linearized}):
\begin{subequations}{\label{E:Good_variables_eq}}
\begin{align}[left = \empheqlbrace\,]
\label{E:s_eq} 
& \Dt s_{2k}
+ \frac{1}{\kappa}\GInvP^{ij} \partial_i r (w_{2k})_i
+ r   \GInvP^{ij}  \partial_i (w_{2k})_j 
+ra_1   v^i \partial_i s_{2k} 
 = f_{2k}
\\
& \Dt (w_{2k})_i + a_2 \partial_i s_{2k}   = (g_{2k})_i, 
\label{E:w_eq}
\end{align}
\end{subequations}
with source terms $(f_{2k},g_{2k})$ which will be shown to be perturbative, see Lemma~\ref{l:fg-2k}.
For later use we compute the expressions for the source terms $(f_{2k},g_{2k})$, which are given by
\begin{subequations}
\begin{align}
\label{E:Source_s_eq}
f_{2k} &= [ r \GInvP^{ij} \partial_i, \Dt^{2k} ] v_j
+ 
[r a_1 v^i \partial_i, \Dt^{2k} ] r 
- 
r  \frac{a_0 a_1}{\kappa \br} \GInvP^{pq} \partial_q r v^i \partial_i \Dt^{2k-1} v_p
\\
& \ \ \ 
- \Dt \left(\frac{a_0}{\kappa \br} \GInvP^{ij} \partial_i r \right) \Dt^{2k-1} v_j 
- r a_ 1  v^i \partial_i \left(\frac{a_0}{\kappa
\br} \GInvP^{pq} \partial_p r \right) \Dt^{2k-1} v_q
\nonumber
\\
& \ \ \ 
-
r a_3 \GInvP^{ij} \partial_i r (w_{2k})_j ,
\nonumber
\\
\label{E:Source_w_eq}
(g_{2k})_i & = \Dt^{2k-1} [ a_2 \partial_i, \Dt] r + 
[a_2 \partial_i, \Dt^{2k-1} ] \Dt r
- 
\frac{a_0 a_2}{\kappa \br}\GInvP^{jl} \partial_j r 
\partial_i \Dt^{2k-1} v_l
\\
& \ \ \
-a_2 \partial_i\left( \frac{a_0}{\kappa
\br }
\GInvP^{ml} \partial_m  r \right) \Dt^{2k-1} v_l,
\nonumber
\end{align}
\end{subequations}
where we used that
\begin{align}
[A,BC] = [A,B]C + B[A,C].
\nonumber
\end{align}
to write 
\begin{align}
[a_2 \partial_i, \Dt^{2k} ] 
= [a_2 \partial_i, \Dt^{2k-1}] \Dt + \Dt^{2k-1} [ a_2 \partial_i, \Dt ].
\nonumber
\end{align}

We also define
\begin{align}
\Vort_{2k} = r^a \partial^b \Vort, \, 
|b| \leq 2k-1, \, 
|b|-a = k-1
\label{E:Vorticity_2k_def},
\end{align}
which we think of as the vorticity counterpart to 
$(s_{2k},w_{2k})$. These we will think of as solving approximate transport equations;
using \eqref{E:Transport_vorticity}
we find
\begin{align}
\begin{split}
\Dt (\Vort_{2k})_{ij} 
+ \frac{1}{v^0} (\partial_i v^l (\Vort_{2k})_{l j} 
+  \partial_j v^l (\Vort_{2k})_{i l } )
- \frac{v^l}{(v^0)^2}  (\partial_iv^0v (\Vort_{2k})_{lj} - \partial_j v^0
(\Vort_{2k})_{li})
&= (h_{2k})_{ij},
\end{split}
\label{E:Transport_vorticity_2k}
\end{align}
where $h_{2k}$ 
is given by
\begin{align}
\begin{split}
    (h_{2k})_{ij} &= [\Dt, r^a \partial^b] \Vort_{ij} 
    + [\frac{1}{v^0} \partial_i v^l ,
    r^a \partial^b] \Vort_{l j}
    + 
    [\frac{1}{v^0}  \partial_j v^l,r^a\partial^b] 
    \Vort_{i l } 
    - [\frac{1}{(v^0)^2}  \partial_iv^0v^l,r^a \partial^b] \Vort_{lj}
    \\
    &
    \quad \ + [\frac{1}{(v^0)^2}  \partial_j v^0 v^l,r^a\partial^b]
    \Vort_{li}.
\end{split}
    \label{E:Source_vort_2k_eq}
\end{align}
We introduce the wave energy
\begin{align}
E^{2k}_{\text{wave}}(r,v) := \sum_{j=0}^{k} 
\norm{ (s_{2j}, w_{2j}) }^2_{\mathcal{H}},
\nonumber
\end{align}
the transport energy 
\begin{align}
E^{2k}_{\text{transport}}(r,v) := \norm{\Vort}^2_{H^{2k-1,k+
\frac{1}{2\kappa}}},
\nonumber
\end{align}
and the total energy
\begin{align}
E^{2k}(r,v) := E^{2k}_{\text{wave}}(r,v) + E^{2k}_{\text{transport}}(r,v) .
\label{E:Total_energy_def}
\end{align}

\subsection{Energy coercivity\label{S:Energy_coercivity}}
Our goal in this section is to show that 
the energy \eqref{E:Total_energy_def}
measures the $\mathbf{H}^{2k}$ size of 
$(r,v)$. To do so, we would like
to consider the energy as a functional
of $(r,v)$ defined at a fixed time.
This can be done by using 
equations \eqref{E:r_v_eq}
to algebraically solve for spatial derivatives
of $(r,v)$.

\begin{theorem}
\label{T:Energy_norm_equivalence}
Let $(r,v)$ be smooth functions
in $\overline{\MovingDom}$. Assume that $r$ is positive in $\MovingDom$ and uniformly non-degenerate
on the $\Gamma$. Then
\begin{align}
E^{2k} \approx_A \norm{ (r,v) }^2_{\mathcal{H}^{2k}}.
\nonumber
\end{align}
\end{theorem}
\begin{proof}
We begin with the $\lesssim$ part. We consider the wave part of the energy
and the corresponding expressions for $(s_{2k},w_{2k})$. Using use equations
\eqref{E:r_eq} and \eqref{E:v_eq} to successively solve for $\Dt(r,v)$, we obtain that each 
$(s_{2k},w_{2k})$ is a linear combination of multilinear expressions in $r$ and 
$\partial v$ (with order zero coefficients).

We will use our bookkeeping scheme of Section \ref{S:Bookkeeping}
to understand the expressions for $(s_{2k},w_{2k})$. It is useful to record
here  the order and structure of the linear-in-derivatives top order
terms obtained by using the equations to inductively compute 
 $\Dt^{2k} (r,v)$ and  $\Dt^{2k-1}(r,v)$,
which involve $2k$ and $2k-1$ derivatives, respectively:
\begin{subequations}{\label{E:Dt_2k_2k+1_leading_order}}
\begin{align}
\Dt^{2k} r &\approx r^k \partial^{2k}r  + r^{k+1} \partial^{2k} v \approx  
r^k \partial^{2k} r,
\label{E:Dt_2k_r_leading_order}
\\
(k-1) & \approx (k-1) + (k-\frac{3}{2}) \approx (k-1),
\nonumber
\\
\Dt^{2k} v & \approx r^k \partial^{2k} v + r^{k} \partial^{2k} r \approx  r^k \partial^{2k} v ,
\label{E:Dt_2k_v_leading_order}
\\
(k-\frac{1}{2}) & \approx (k-\frac{1}{2}) + (k-1) \approx (k-\frac{1}{2}) ,
\nonumber
\\
\Dt^{2k-1} r &\approx r^{k} \partial^{2k-1}r  + r^{k} \partial^{2k-1} v \approx 
r^{k} \partial^{2k-1} v,
\label{E:Dt_2k-1_r_leading_order}
\\
(k-\frac{3}{2}) & \approx  (k-2) + (k-\frac{3}{2}) \approx  (k-\frac{3}{2}),
\nonumber
\\
\Dt^{2k-1} v & \approx r^{k} \partial^{2k-1} v + r^{k-1} \partial^{2k-1} r \approx 
 r^{k-1} \partial^{2k-1} r,
\label{E:Dt_2k-1_v_leading_order}
\\
(k -1 )& \approx (k-\frac{3}{2}) + (k-1) \approx (k-1).
\nonumber
\end{align}
\end{subequations}
Expressions \eqref{E:Dt_2k_2k+1_leading_order} are obtained by successively solving for 
$\Dt(r,v)$ in \eqref{E:r_eq}-\eqref{E:v_eq}. Below each expression in 
\eqref{E:Dt_2k_r_leading_order}-\eqref{E:Dt_2k-1_v_leading_order} we have written the orders
of the corresponding terms. The terms of order
$k-3/2$, $k-1$, $k-2$, and $k-3/2$ in \eqref{E:Dt_2k_r_leading_order},
\eqref{E:Dt_2k_v_leading_order}, \eqref{E:Dt_2k-1_r_leading_order}, and \eqref{E:Dt_2k-1_v_leading_order},
respectively have orders less than the other terms in the same expressions, despite
having the same number of derivatives, and hence are dropped in the second $\approx$ on each line.
Such terms have smaller order, even though they have the same number of derivatives, 
because of extra powers of $r$, and come from the term
$r a_1 v^i \partial_i r$ in \eqref{E:r_eq}.

We begin with the expressions of highest order (see Section \ref{S:Bookkeeping}), thus we first 
focus on the multilinear expressions that come from ignoring
the last term on LHS \eqref{E:r_eq} and also where no derivative lands on $\GInv$, $a_1$, and $a_2$. 
We also consider first the case when every time we commute $\Dt$ with $\partial$, the derivative
lands on $v^i$ and not on $r$ via $v^0$. 

In this case,
the corresponding multilinear expressions for $(s_{2k},w_{2k})$ have the following properties:
\begin{itemize}
\item They have order $k-1$ and $k-\frac{1}{2}$, respectively.
\item They have exactly $2k$ derivatives.
\item They contain at most $k+1$ and $k$ factors of $r$, respectively.
\end{itemize}
Thus, a multilinear expression for $s_{2k}$ in this case has the form
\begin{align}
M = r^a \prod_{j=1}^J \partial^{n_j} r \prod_{l=1}^L \partial^{m_l} v,
\label{E:General_multilinear_expression_good_s}
\end{align}
where $n_j,m_l \geq 1$, and subject to
\begin{equation}\label{count-at-scale}
\begin{aligned}
&\sum n_j + \sum m_l = 2k,
\\
&a + J + L/2 = k+1,
\end{aligned}
\end{equation}
and when $J=0$ or $L=0$ the corresponding product is omitted. We claim that it is possible to 
choose $b_j$ and $c_l$ such that 
\begin{align}
0 \leq b_j \leq (n_j - 1 ) \frac{k}{2k-1}, \, 0 \leq c_l \leq (m_l - 1) \frac{k+1/2}{k - 1/2}, \,
a = \sum b_j + \sum c_l.
\nonumber
\end{align}
This follows from observing that
\begin{align}
& \sum (n_j - 1 ) \frac{k}{2k-1} + \sum  (m_l - 1) \frac{k+1/2}{k - 1/2}
\leq 
\left( \sum n_j + \sum m_l - J - L \right) \frac{k}{2k-1} 
\nonumber
\\
&= (2k-J-L) \frac{k}{2k-1} \leq (a+k-1) \frac{k}{2k-1} \leq a,
\nonumber
\end{align}
since $a\leq k$. Equality holds only if $a=k$, $J=1$ and $L=0$ (i.e., for the leading linear case).
This shows that it is possible to make such a choice of $b_j$ and $c_l$, which allows us to 
use our interpolation theorems
\begin{align}
\norm{r^{b_j} \partial^{n_j} r}_{L^{p_j}(r^{\frac{1-\kappa}{\kappa}})} & \lesssim (1+A)^{1-\frac{2}{p_j}} 
\norm{(r,v)}_{\mathcal{H}^{2k}}^{\frac{2}{p_j}},
\nonumber
\\
\norm{r^{c_l} \partial^{m_l} v }_{L^{q_l}(r^{\frac{1-\kappa}{\kappa}})} & \lesssim A^{1-\frac{2}{q_l}} 
\norm{(r,v)}_{\mathcal{H}^{2k}}^{\frac{2}{q_l}},
\nonumber
\end{align}
where
\begin{align}
\frac{1}{p_j} = \frac{n_j -1 -b_j}{2(k-1)}, \, 
\frac{1}{q_l} = \frac{m_l -1/2 -c_l}{2(k-1)}.
\nonumber
\end{align}
Observe that the numerators in $1/p_j$ and $1/q_l$ correspond to the orders
of the expressions being estimated and they add up to $k-1$ (as needed). 

In addition to the principal part discussed above, we also obtain lower order terms
in our expression for $s_{2k}$. There are three sources of such terms:

\begin{enumerate}[label=\roman*)]
    \item  The terms from the commutator $[\Dt,\partial]$ where derivatives apply to $r$  via $v^0$. This corresponds to the second term in the formal expansion
\begin{align}
[\Dt, \partial] \approx (\partial v )\partial + (\partial r) \partial,
\nonumber
\end{align}
whose order is easily seen to be $1/2$ lower. 

\item If any derivatives are applied to either $r$ or $v$ via $a_0$, $a_1$, $a_2$ or $G$,
this increases the order of the resulting expression by $0$, respectively $1/2$, compared 
to the full order of the derivative which is $1$.

\item Contributions arising from the last term in \eqref{E:r_def},
whose order is, to start with, $1/2$ lower than the rest of the terms in the \eqref{E:r_def}.
\end{enumerate}

The contributions of all such terms to $s_{2k}$ have lower order. More 
precisely, they contain expressions of the form \eqref{E:General_multilinear_expression_good_s}
but with \eqref{count-at-scale} replaced by
\begin{equation}\label{count-below-scale}
\begin{aligned}
\sum n_j + \sum m_l &= 2k,
\\
a + J + L/2 &= k+1+\frac{j}{2},
\end{aligned}
\end{equation}
where $j > 0$, and which have lower order $k-1-\frac{j}{2}$.
All such lower order terms can be estimated in a similar fashion, but using lower Sobolev norms for $(r,v)$.

\bigskip
We continue with the $\gtrsim$ part. 
Applying $\Dt$ to \eqref{E:s_eq} and \eqref{E:w_eq} and using definitions
\eqref{E:Good_variables_def}
we find the following recurrence relations
\begin{subequations}{\label{E:Recurrence_equations}}
\noeqref{E:Recurrence_s,E:Recurrence_w}
\begin{align}
 s_{2j} & = \tilde{L}_1 s_{2j-2}
  + F_{2j},
\label{E:Recurrence_s}
\\
(w_{2j})_i & = (\tilde{L}_2 w_{2j-2} )_i + G_{2j},
\label{E:Recurrence_w}
\end{align}
\end{subequations}
where $\tilde{L}_1$ and $\tilde{L}_2$ have been defined in
Section \ref{S:Linearized}.
The next Lemma characterizes
the error terms on the RHS of \eqref{E:Recurrence_equations},
and the lemma that follows gives a quantitative 
relation between the $2j$ and $2j-2$ quantities.

\begin{lemma}
\label{L:Recurrence_relations}
For $j\geq 2$, the terms $F_{2j}$ and $G_{2j}$
in \eqref{E:Recurrence_equations}
are linear combinations of multilinear expressions
in $r$ and $\partial v$ with
$2j$ derivatives and of order at most $j-1$ and $j-\frac{1}{2}$, respectively.
Moreover, they are either
\begin{enumerate} [label= \roman*)]
    \item non-endpoint, by which we mean 
multilinear expressions 
of order $j-1$ and $j-\frac{1}{2}$, respectively,
containing at most $j+1$ and $j$ factors of $r$, respectively,
and whose products contain
at least two factors of $\partial^{\geq 2} r$
or $\partial^{\geq 1} v$, or 
\item they have order
strictly less than 
$j-1$ and $j-\frac{1}{2}$, respectively, and contain at most
$j+2$ and $j+1$ factors of $r$, respectively.
\end{enumerate}
\end{lemma}
\begin{proof}
We begin with $j\geq 3$. We will analyze 
\begin{align}
    s_{2j} &= \Dt^{2j} r - \frac{a_0}{\kappa
    \br}
    \GInvP^{lm}\partial_l r \Dt^{2j-1} v_m.
    \label{E:s_2j_explicit}
\end{align}
In order to keep track of terms according to the statement of the Lemma,
we observe that $s_{2j}$ has order $j-1$. We will make successive
use of the commutator
\begin{align}
    \begin{split}
        [\Dt,\partial_l]
        = -\frac{a_0}{\kappa
        \br}
        \GInvP^{pq} \partial_l v_q \partial_p 
        +
        \frac{\br^{1+\frac{2}{\kappa}}}{(v^0)^3}
        v^p \partial_l r \partial_p.
    \end{split}
    \nonumber
\end{align}

We begin with the first term on RHS \eqref{E:s_2j_explicit}.
From \eqref{E:r_v_eq}, we compute.
\begin{align}
    \begin{split}
    \Dt^2 r &= r\GInvP^{ml} \partial_l (a_2 \partial_m r) 
    -[\Dt, r\GInvP^{ml} \partial_l  ] v_m 
    - [\Dt, r a_1 v^l \partial_l] r
    + r a_1 v^i \partial_i \left( r\GInvP^{ml} \right)\partial_l v_m
    \\
    & \ \ \
    + r a_1 v^i \partial_i (r a_1 v^l \partial_l )r
    + r^2 a_1^2 v^l v^m \partial_l \partial_m r
    + r^2 a_1 \GInvP^{ml} v^i \partial_i \partial_l v_m.
    \end{split} 
    \nonumber 
\end{align}
Then,
\begin{align}
    \begin{split}
        \Dt^{2j} r & = \Dt^{2j-2} \Dt^2 r
        \\
        &= 
        \Dt^{2j-2} \Big( 
        r\GInvP^{ml} \partial_l (a_2 \partial_m r) 
    -[\Dt, r\GInvP^{ml} \partial_l  ] v_m 
    - [\Dt, r a_1 v^l \partial_l] r
    \\
    & \ \ \
    + r a_1 v^i \partial_i \left( r\GInvP^{ml} \right)\partial_l v_m
    + r a_1 v^i \partial_i (r a_1 v^l \partial_l )r
    \\
    & \ \ \ 
    + r^2 a_1^2 v^l v^m \partial_l \partial_m r
    + r^2 a_1 \GInvP^{ml} v^i \partial_i \partial_l v_m,
    \Big).
    \end{split}
    \label{E:s_2j_first_term_expanded}
\end{align}
and we will consider each term on RHS \eqref{E:s_2j_first_term_expanded}
separately. 

The terms $ \partial_i \left( r\GInvP^{ml} \right)$
and $\partial_i (r a_1 v^l \partial_l )r$ have order at most zero, thus
\begin{align}
    \Dt^{2j-2} \left( r a_1 v^i \partial_i \left( r\GInvP^{ml} \right)\partial_l v_m
    + r a_1 v^i \partial_i (r a_1 v^l \partial_l )r
    \right)
    \nonumber
\end{align}
has order at most $j-3/2$ and belongs to $F_{2j}$.
Next,
\begin{align}
     \begin{split}
         [\Dt,r\GInvP^{ml} \partial_l] v_m 
         = 
         [\Dt,r\GInvP^{ml} ] \partial_l v_m 
         +
         r\GInvP^{ml} [\Dt,\partial_l]   v_m .
     \end{split}
     \nonumber
\end{align}
For the first term on the RHS, we have
\begin{align}
    \begin{split}
        [\Dt,r\GInvP^{ml}]  \partial_lv_m  & =
        \Dt r \GInvP^{lm} \partial_l v_m
        + r \frac{\partial \GInvP^{lm}}{\partial r} \Dt r \partial_l v_m
        + r \frac{\partial \GInvP^{lm}}{\partial v_i} \Dt v_i
        \partial_l v_m.
    \end{split}
    \nonumber
\end{align}
The second and third terms have order $\leq -1$ and $-1/2$, respectively,
thus they belong to $F_{2j}$ after differentiation
by $\Dt^{2j-2}$. For the first term, we have
\begin{align}
    \Dt r \GInvP^{lm} \partial_l v_m
    & = 
    - r \GInvP^{ij} \partial_i v_j \GInvP^{lm} \partial_l v_m
    - r a_1 v^i \partial_i r \GInvP^{lm} \partial_l v_m.
    \nonumber
\end{align}
The first term satisfies the non-endpoint property while the 
second has order $-1/2$, thus both terms belong to $F_{2j}$ after
differentiation by $\Dt^{2j-2}$. Next,
\begin{align}
    \begin{split}
        r \GInvP^{ml} [\Dt, \partial_l] v_m &
        = -\frac{ r a_0}{\kappa( 1 + \br}
        \GInvP^{pq} \partial_l v_q  \GInvP^{ml} \partial_p v_m
        +
        \frac{r \br^{1+\frac{2}{\kappa}}}{(v^0)^3}\GInvP^{ml}
        v^p \partial_l r \partial_p v_m.
    \end{split}
    \nonumber
\end{align}
The second term has order $-1/2$ so it belongs to $F_{2j}$ upon
differentiation by $\Dt^{2j-2}$. The first term has order zero, 
thus producing a top order (i.e., $j-1$) term when differentiated
by $\Dt^{2j-2}$. Nevertheless, it has two $\partial^{\geq 1}v$ terms
so it satisfies the non-endpoint property and hence it also belongs
to $F_{2j}$.

We now turn to the other commutator in \eqref{E:s_2j_first_term_expanded}:
\begin{align}
    \begin{split}
        [\Dt, r a_1 v^l \partial_l] r & = 
        [\Dt, r a_1 v^l] \partial_l 
        r
        + r a_1 v^l [\Dt,\partial_l] r
        \\
        & = 
        \Dt r a_1 v^l \partial_l r
        + r \frac{\partial a_1}{\partial r} \Dt r 
        v^l
        \partial_l r
        + r \frac{\partial a_1}{\partial v_i} \Dt v_i v^l  \partial_l r
        + r a_1 \Dt v^l \partial_l r
        \\
        & \ \ \ 
         -\frac{r a_0 a_1 }{\kappa \br}
        \GInvP^{pq} v^l \partial_l v_q \partial_p r
        +
        \frac{r a_1 \br^{1+\frac{2}{\kappa}}}{(v^0)^3}
        v^p v^l \partial_l r \partial_p r.
    \end{split}
    \nonumber 
\end{align}
The terms on the RHS have orders $\leq -1/2,-3/2,-1,-1,-1/2,-1$, respectively, so they all belong to $F_{2j}$ upon differentiation
by $\Dt^{2j-2}$.

For the last two terms on RHS \eqref{E:s_2j_first_term_expanded},
\begin{align}
     r^2 a_1^2 v^l v^m \partial_l \partial_m r
    + r^2 a_1 \GInvP^{ml} v^i \partial_i \partial_l v_m,
    \nonumber
\end{align}
we see that they have orders $-1$ and $-1/2$, thus also belong to 
$F_{2j}$ after differentiation
by $\Dt^{2j-2}$.

Therefore, writing $\approx$ for equality modulo terms that belong to
$F_{2j}$, \eqref{E:s_2j_first_term_expanded} becomes
\begin{equation}
\begin{aligned}
        \Dt^{2j} r & = \Dt^{2j-2} \Dt^2 r
        \\
        &\approx 
        \Dt^{2j-2} \Big( 
        r\GInvP^{ml} \partial_l (a_2 \partial_m r) 
        \Big)
        \\
        &
        = \sum_{\ell=0}^{2j-2} \binom{2j-2}{\ell}
        \Dt^{2j-2-\ell} r \Dt^\ell 
        \Big( 
        \GInvP^{ml} \partial_l (a_2 \partial_m r) 
        \Big)
        \\
        &= 
        r 
        \Dt^{2j-2}
        \Big( 
        \GInvP^{ml} \partial_l (a_2 \partial_m r) 
        \Big)
        + 
        \sum_{\ell=0}^{2j-2-1} \binom{2j-2}{\ell}
        \Dt^{2j-2-\ell} r \Dt^\ell 
        \Big( 
        \GInvP^{ml} \partial_l (a_2 \partial_m r)
        \Big).
    \nonumber
\end{aligned}
\end{equation}
In the second sum, we can further write
\begin{align}
    \begin{split}
        \Dt^{2j-2-\ell} r \Dt^\ell 
        \Big( \GInvP^{ml} \partial_l (a_2 \partial_m r) \Big)
        & =
            \Dt^{2j-2-\ell} r \Dt^\ell 
        \Big( \GInvP^{ml} a_2  \partial_l  \partial_m r
        + 
        \GInvP^{ml} \partial_l a_2 \partial_m r
        \Big)
    \end{split}
    \nonumber
\end{align}
The term $\Dt^{2j-2-\ell} r \Dt^\ell  
(
\GInvP^{ml} \partial_l a_2 \partial_m r
)$ has order at most $j-3/2$ and can be absorbed
into $F_{2j}$. For the first term, if $\Dt^\ell$ hits 
$\GInvP^{ml} a_2$ we again obtain a term of order strictly 
less than $j-1$ that is part of $F_{2j}$. Finally, for the term
\begin{align}
    \Dt^{2j-2-\ell} r 
    \GInvP^{ml} a_2  \Dt^\ell  \partial_l  \partial_m r,
    \nonumber
\end{align}
we use that $\ell \leq 2j-2-1$ and \eqref{E:r_eq}
to write
\begin{align}
    \begin{split}
    \Dt^{2j-2-\ell} r 
    \GInvP^{ml} a_2  \Dt^\ell  \partial_l  \partial_m r
    & =
    -\GInvP^{ml} a_2  \Dt^\ell  \partial_l  \partial_m r \Dt^{2j-3-\ell} \left( r\GInvP^{pq} \partial_p v_q \right)
    \\
    & \ \ \ 
    -\GInvP^{ml} a_2  \Dt^\ell  \partial_l  \partial_m r
    \Dt^{2j-3-\ell} \left( r a_1 v^p \partial_p r \right).
    \end{split}
    \nonumber
\end{align}
The first term contains a $\partial^{\geq 1} v$ and 
a $\partial^{\geq 2}r $ so it belongs to $F_{2j}$, whereas
the second term has order at most $j-3/2$ so it belongs to 
$F_{2j}$ as well. Hence, we have that
\begin{align}
    \begin{split}
            \Dt^{2j} r & \approx
             r 
        \Dt^{2j-2}
        \Big( 
        \GInvP^{ml} \partial_l (a_2 \partial_m r) 
        \Big) 
        \\
        & = r\Dt^{2j-3} \Big (
        \GInvP^{ml} \partial_l (a_2 \partial_m \Dt r ) \Big )
        + r\Dt^{2j-3} \Big ( [\Dt, \GInvP^{lm} \partial_l (a_2 \partial_m \cdot ) ] r \Big ).
    \end{split}
    \nonumber
\end{align}
We now compute the commutator on the second term on the RHS:
\begin{align}
    \begin{split}
        [\Dt, \GInvP^{lm} \partial_l (a_2 \partial_m \cdot ) ] r &=
        \Dt \left( \GInvP^{lm} \partial_l (a_2 \partial_m r) \right)
        - \GInvP^{lm} \partial_l \left( a_2 \partial_m \Dt r \right)
        \\
        &=
        \frac{\partial \GInvP^{lm}}{\partial r}\Dt r \partial_l(a_2\partial_m r) 
        +
         \frac{\partial \GInvP^{lm}}{\partial v_i}\Dt v_i \partial_l(a_2\partial_m r)    
        \\
        & \ \ \ 
        + \GInvP^{lm} \Dt \partial_l (a_2 \partial_m r) 
        - \GInvP^{lm} \partial_l(a_2 \partial_m \Dt r).
    \end{split}
    \nonumber
\end{align}
The first and second terms on the RHS of the second equality have orders
$\leq 1/2$ and $1$, respectively, so they produce terms of order
at most $j-3/2$ when hit by $r\Dt^{2j-3}$ and thus can be discarded.
Continuing
\begin{align}
    \begin{split}
        [\Dt, \GInvP^{lm}\partial_l (a_2 \partial_m \cdot ) ] r &\approx
        \GInvP^{lm} \Dt \partial_l (a_2 \partial_m r) 
        - \GInvP^{lm} \partial_l(a_2 \partial_m \Dt r)
        \\
        & = 
        \GInvP^{lm}\Big( a_2 \Dt \partial_l \partial_m r 
        - a_2 \partial_l \partial_m \Dt r \Big )
        \\
        & 
        \ \ \ +
        \GInvP^{lm} \Big( \Dt \partial_l a_2 \partial_m r
        + \partial_l a_2 \Dt \partial_m r 
        + \Dt a_2 \partial_m \partial_l r 
        - \partial_l a_2 \partial_m \Dt r \Big ).
    \end{split}
    \nonumber
\end{align}
All the terms inside the second parenthesis have orders at most $1$
(thus giving order at most $j-3/2$ when hit by $r\Dt^{2j-3}$) and
can be discarded. The terms in the first parenthesis 
give $\GInvP^{lm} a_2 [\Dt, \partial_l \partial_m]r$. Continuing
\begin{align}
    \begin{split}
         [\Dt, \GInvP^{lm}\partial_l (a_2 \partial_m \cdot ) ] r &\approx
        \GInvP^{lm} a_2 [\Dt, \partial_l \partial_m]r
        =
        \GInvP^{lm} a_2 \Big ( [\Dt,\partial_l] \partial_m r
        + \partial_l \left( [\Dt, \partial_m] r \right) \Big)
        \\
        & =
        \GInvP^{lm} a_2 \Big ( 
        -\frac{a_0}{\kappa \br} \GInvP^{pq} 
        \partial_l v_q \partial_p \partial_m r 
        + \frac{\br^{1+\frac{2}{\kappa}}}{(v^0)^3}
        v^p \partial_l r \partial_p \partial_m r \Big )
        \\
        &\quad \  + 
        \GInvP^{lm} a_2 \partial_l \Big ( 
        -\frac{a_0}{\kappa \br} \GInvP^{pq} 
        \partial_m v_q \partial_p  r 
        + \frac{\br^{1+\frac{2}{\kappa}}}{(v^0)^3}
        v^p \partial_m r \partial_p  r \Big ).
    \end{split}
    \nonumber 
\end{align}
The second term in the first parenthesis has order $1$. The second
term in the second parenthesis produces, after differentiation by 
$\partial_l$, terms of order at most $1$. Hence, the second
terms in both parenthesis give order at most $j-3/2$ after we apply
$r\Dt^{2j-3}$ and belong to $F_{2j}$. Moreover, when $\partial_l$ in 
front of the second parenthesis hits the zero order coefficients
in the first term it gives terms of order at most $1$ which 
can again be discarded; when it hits $\partial_p r$ it produces a term that
can be combined with the first term in the first parenthesis.
Therefore, we have
\begin{align}
    \begin{split}
        \Dt^{2j} r & \approx 
        r\Dt^{2j-3} \Big ( 
        \GInvP^{ml} \partial_l (a_2 \partial_m \Dt r ) \Big )
        -r \Dt^{2j-3} 
        \Big ( 
        \frac{a_0 a_2}{\kappa \br} \GInvP^{lm}
        \partial_l \partial_m v_q \GInvP^{pq} \partial_p r 
        \Big)
        \\
        & \ \ \ 
        -2 r \Dt^{2j-3} 
        \Big (
         \frac{a_0 a_2}{\kappa \br} \GInvP^{lm}
         \GInvP^{pq} \partial_m v_q \partial_p \partial_l r
         \Big ).
    \end{split}
    \label{E:Dt_2j_r_main_terms}
\end{align}
The last term on RHS \eqref{E:Dt_2j_r_main_terms}
has a $\partial v \partial^2 r$ factor. Hence it produces,
after application of $r \Dt^{2j-3}$ either non-endpoint terms
or terms of order $< j-1$, so it belongs to $F_{2j}$.

We now analyze the second term on RHS \eqref{E:Dt_2j_r_main_terms}.
We distribute $\Dt^{2j-3}$. Whenever at least one $\Dt$ hits one 
of the zero order factors it results in a term of order $\leq j-3/2$
that can be absorbed into $F_{2j}$. Hence we are left with
\begin{align}
    \begin{split}
        & -r
        \frac{a_0 a_2}{\kappa \br} \GInvP^{lm}\GInvP^{pq}
         \Dt^{2j-3} 
         \Big ( 
         \partial_l \partial_m v_q  \partial_p r 
        \Big) 
 =
         - r
        \frac{a_0 a_2}{\kappa \br} \GInvP^{lm}\GInvP^{pq}
         \sum_{\ell=0}^{2j-3} \binom{2j-3}{\ell} 
         \Dt^{2j-3-\ell} \partial_l \partial_m v_q
         \Dt^\ell \partial_p r.
    \end{split}
    \nonumber
\end{align}
The terms in the sum with $l \neq 0$ belong to $F_{2j}$. For, after commuting
$\Dt$ with $\partial$, we obtain
either lower order terms or $\partial v \partial^2 r$ factors, so we
are left with
\begin{align}
    \begin{split}
        -r
        \frac{a_0 a_2}{\kappa \br} \GInvP^{lm}\GInvP^{pq}
         \Dt^{2j-3} \partial_l \partial_m v_q
         \partial_p r
         =&\,
        -r
        \frac{a_0 a_2}{\kappa \br} \GInvP^{lm}\GInvP^{pq}
         \partial_l \partial_m \Dt^{2j-3}  v_q 
         \partial_p r
         \\
         &\,
        -r
        \frac{a_0 a_2}{\kappa \br} \GInvP^{lm}\GInvP^{pq}
        [\Dt^{2j-3}, \partial_l \partial_m]  v_q
         \partial_p r.
    \end{split}
    \nonumber 
\end{align}
The first term on the RHS belongs to $\tilde{L}_1 s_{sj-2}$. The second
term on the RHS belongs to $F_{2j}$. This can be seen by computing
the commutator in similar fashion to what we did
to compute $[\Dt, \GInvP^{lm}\partial_l (a_2 \partial_m \cdot ) ]$ 
(in fact, $[\Dt^{2j-3}, \GInvP^{lm}\partial_l (a_2 \partial_m \cdot ) ] r$ and
$[\Dt^{2j-3}, \partial_l \partial_m]$ are the same modulo
lower terms).

It remains to analyze the first term on RHS \eqref{E:Dt_2j_r_main_terms}.
We have 
\begin{align}
    \begin{split}
        r\Dt^{2j-3} \Big ( 
        \GInvP^{ml} \partial_l (a_2 \partial_m \Dt r ) \Big )
        & =
         \GInvP^{ml} \partial_l (a_2 \partial_m \Dt^{2j-2} r )
         +
         r
         [\Dt^{2j-3}, \GInvP^{lm}\partial_l (a_2 \partial_m \cdot ) ]
         \Dt r.
    \end{split}
    \nonumber 
\end{align}
The first term on the RHS belongs to $\tilde{L}_1 s_{2j-2}$.
The term of order $j-1$ from the second term on the RHS is non-endpoint, as 
it comes from combining $\partial v$ from the commutator with $\partial v$
from $\Dt r$.

We next consider the second term on \eqref{E:s_2j_explicit}. We have
\begin{align}
    \begin{split}
         - \frac{a_0}{\kappa \br} \GInvP^{lm} \partial_l r
        \Dt^{2j-1} v_m  = 
        & - \frac{a_0}{\kappa \br} \GInvP^{lm} \partial_l r
        \Dt^{2j-3} \Dt (-a_2 \partial_m r)
        \\
        = & \
        \frac{a_0}{\kappa \br} \GInvP^{lm} \partial_l r
        \Dt^{2j-3}\Big( a_2 \partial_m \Dt r + [\Dt, a_2 \partial_m] r
        \Big).
    \end{split}
    \label{E:s_2j_explicit_second_term}
\end{align}
Consider the second term on RHS  \eqref{E:s_2j_explicit_second_term}.
Using arguments similar to above, we can show that all terms
belong to $F_{2j}$, except for the term that corresponds
to all $\Dt^{2j-3}$ hitting the $\partial v$ from the commutator
$[\Dt,\partial_m]$, i.e., except for
\begin{align}
    \begin{split}
   - a_2\left(\frac{a_0}{\kappa \br} \right)^2 
   \GInvP^{lm} \partial_l r
   \GInvP^{pq} \Dt^{2j-3} \partial_m  v_q \partial_p r
   =&- a_2\left(\frac{a_0}{\kappa \br} \right)^2 
   \GInvP^{lm} \partial_l r
        \GInvP^{pq}  \partial_m \Dt^{2j-3} v_q \partial_p r
    \\
    &
    - a_2
    \left(\frac{a_0}{\kappa \br} \right)^2 
   \GInvP^{lm} \partial_l r
        \GInvP^{pq} [\Dt^{2j-3}, \partial_m] v_q \partial_p .
    \end{split}
    \nonumber 
\end{align}
The commutator term can again be shown to belong to $F_{2j}$ using the
same sort of calculations as above. Modulo terms that can be
absorbed into $F_{2j}$, the remaining term can be written as
\begin{align}
    \begin{split}
        a_2 \frac{a_0}{\kappa \br}
        \GInvP^{lm} \partial_l r\partial_m
        \Big(
        -\frac{a_0}{\kappa \br}
        \GInvP^{pq}   \Dt^{2j-3} v_q \partial_p r
        \Big)
        = &\,
        \frac{ a_2}{\kappa}
        \GInvP^{lm} \partial_l r\partial_m
        \Big(
        -\frac{a_0}{\kappa \br}
        \GInvP^{pq}   \Dt^{2j-3} v_q \partial_p r
        \Big)
        \\
        & +
        r a_3
        \GInvP^{lm} \partial_l r\partial_m
        \Big(
        -\frac{a_0}{\kappa \br}
        \GInvP^{pq}   \Dt^{2j-3} v_q \partial_p r
        \Big),
    \end{split}
    \nonumber
\end{align}
where we used \eqref{E:a_3_def_and_relation}. The first term
on the RHS belongs to $\tilde{L}_1 s_{2j-2}$ and the second
one can be absorbed into $F_{2j}$.

The first term on RHS \eqref{E:s_2j_explicit_second_term}
is treated with similar ideas. We notice that the
top order term in that expression is 
\begin{align}
    \begin{split}
        & \frac{a_0}{\kappa \br} 
        \GInvP^{lm} \partial_l r
       a_2 \partial_m  \Dt^{2j-2}  r 
       =\,
       \frac{a_2}{\kappa} \GInvP^{lm} \partial_l r
        \partial_m  \Dt^{2j-2}  r
       +ra_3 \GInvP^{lm} \partial_l r
        \partial_m  \Dt^{2j-2}  r .       
    \end{split}
    \nonumber
\end{align}
The first term belongs to $\tilde{L}_1 s_{sj-2}$ and the second
one to $F_{2j}$.

The case $j=2$ is done separately (since the definition of $s_2$ is different, recall \eqref{E:Good_variables_def}), but it 
follows essentially the same steps as above. Finally, 
the proof for $G_{2j}$ is done with the same type of calculations
employed above and we omit it for the sake of brevity.
\end{proof}

To continue our analysis, we need some coercivity estimates for the $\tilde L_1$,
respectively $\tilde L_2+ \tilde L_3$.
\begin{lemma}
\label{L:Elliptic_estimates}
Assume that $A$ is small. Then 
\begin{subequations}
\begin{align}
\norm{s}_{H^{2,\frac{1}{2\kappa} + \frac{1}{2}}} & \lesssim \norm{\tilde{L}_1 s}_{H^{0,\frac{1}{2\kappa} - \frac{1}{2}}}
+\norm{s}_{L^2(r^\frac{1-\kappa}{\kappa})},
\label{coercive-L1}
\\
\norm{w}_{H^{2,\frac{1}{2\kappa}+1}} & \lesssim \norm{(\tilde{L}_2+\tilde L_3)w}_{H^{0,\frac{1}{2\kappa}} } 
+  \norm{w}_{L^2( r^\frac{1}{\kappa})}.
\label{coercive-L23}
\end{align}
\end{subequations}
\end{lemma}

Here we remark that the lower order terms on the right play no role in the proof, and can be omitted  if $(s,w)$ are assumed to have small support (by Poincare's inequality), or if we use homogeneous norms on the left. 

As a consequence of the second estimate above, we have
\begin{corollary}
Assume that $A$ is small. Then
\[
\norm{w}_{H^{2,\frac{1}{2\kappa}+1}}  \lesssim \norm{\tilde{L}_2 w}_{H^{0,\frac{1}{2\kappa}} } + \norm{\curl w}_{H^{1,\frac{1}{2\kappa}+1}}
+  \norm{w}_{L^2( r^\frac{1}{\kappa})}.
\]
\end{corollary}

In Section~\ref{S:Regular_solutions} will also need the following straightforward alternative form of the 
above result:

\begin{corollary}
Assume that $B$ is small. Then the same result as in Lemma~\ref{L:Elliptic_estimates}
holds for the operators $L_1$, respectively $L_2+L_3$.
\end{corollary}

Here the smallness condition on $B$ allows us to treat the differences  $\tilde L_1 - L_1$, $\tilde L_2 - L_2$, $\tilde L_3 - L_3$ perturbatively.

\begin{proof}
We start with two simple observations. First of all, using a partition of unity 
one can localize the estimates to a small ball. We will assume this is done, and further
we will consider the interesting case where this ball is around a boundary point $x_0$; the analysis is standard elliptic otherwise. We can assume that at $x_0$ on the boundary  we have $\nabla r(x_0) = e_n$ so that in our small ball we have
\begin{align}
|\nabla r - e_n| \lesssim A \ll 1.
\label{E:Derivative_speed_near_boundary_e_n}
\end{align}

Secondly,  the smallness condition on $A$
guarantees that the coefficients $G$ and $a_2$ have a small variation in a small 
ball, and we can freeze these coefficients modulo perturbative errors. Hence, 
we will simply freeze them, and assume that $a_2$ and $G$ are constant. Then $a_2$
only plays a multiplicative role, and will be set to $1$ for the rest of the argument.

A preliminary step in the proof is to observe that we have  the weaker bounds
\begin{subequations}
\begin{align}
\norm{s}_{H^{2,\frac{1}{2\kappa} + \frac{1}{2}}} & \lesssim \norm{\tilde{L}_1 s}_{H^{0,\frac{1}{2\kappa} - \frac{1}{2}}} + \norm{s}_{H^{1,\frac{1}{2\kappa} - \frac{1}{2}}},
\nonumber
\\
\norm{w}_{H^{2,\frac{1}{2\kappa} + 1 }} & \lesssim \norm{(\tilde{L}_2 +\tilde L_3)w}_{H^{0,\frac{1}{2\kappa}  }}
+
\norm{w}_{H^{1,\frac{1}{2\kappa}}}.
\nonumber
\end{align}
\end{subequations}
These bounds can be proved in a standard elliptic fashion by integration by parts, e.g. in the case of the first bound one simply starts with the integral representing $\| \tilde L_1 s \|_{H^{0,\frac{1}{2\kappa}}}^2$ and exchange derivatives between the two factors.
The details are left for the reader.

In view of the above bounds, it suffices to show that
\begin{subequations}
\begin{align}
&\norm{s}_{H^{1,\frac{1}{2\kappa} -\frac{1}{2} }}  \lesssim \norm{\tilde{L}_1 s}_{H^{0,\frac{1}{2\kappa}- \frac{1}{2}} }
+\norm{s}_{L^2(r^\frac{1-\kappa}{2\kappa})},
\label{E:Coercive_estimate_s_short}
\\
&\norm{w}_{H^{1,\frac{1}{2\kappa} }}  \lesssim \norm{(\tilde{L}_2+\tilde L_3) w}_{H^{0,\frac{1}{2\kappa} }} 
+  \norm{w}_{L^2( r^\frac{1}{\kappa})}.
\label{E:Coercive_estimate_w_short}
\end{align}
\end{subequations}

For \eqref{E:Coercive_estimate_s_short}, compute
\begin{align}
\begin{split}
\int_{\MovingDomT} 
r^{\frac{1-\kappa}{\kappa}}
\partial_n s \tilde{L}_1 s \, dx 
& =  \int_{\MovingDomT} 
r^{\frac{1-\kappa}{\kappa}}
 \partial_n s
  \GInvP^{ij} a_2 \left( r \partial_i  \partial_j s
 + \frac{1}{\kappa}  \partial_i r \partial_j s  \right)
\, dx
 \\
&= -\frac{1}{2} \int_{\MovingDomT} r^\frac{1}{\kappa} a_2 
\partial_n \left( \GInvP^{ij}  \partial_i s \partial_j s\right)  \, dx
+ \frac{1}{2}  \int_{\MovingDomT} r^\frac{1}{\kappa} a_2 
\partial_n  \GInvP^{ij}  \partial_i s \partial_j s  \, dx
\\
& \  \quad -
\int_{\MovingDomT}   r^\frac{1}{\kappa} \partial_n s  
\partial_i(a_2  \GInvP^{ij})  \partial_j s \, dx
\\
& \gtrsim 
\int_{\MovingDomT} r^\frac{1 - \kappa}{\kappa} a_2 
 \GInvP^{ij}  \partial_i s \partial_j s  \, dx
 + 
 \int_{\MovingDomT} r r^\frac{1 - \kappa}{\kappa}  |\partial s|^2\, dx.
\end{split}
\nonumber
\end{align}
which suffices by the Cauchy-Schwarz inequality.

Now we consider \eqref{coercive-L23}. As discussed above, we set $a_2=1$
and assume $G$ is a constant matrix. We recall that 
 $\tilde{L}_2$ has the form
\begin{align}
    \begin{split}
    (\tilde{L}_2 w)_i &= 
     G^{ml} \left( \partial_i(r\partial_m w_l) 
    + \frac{1}{\kappa} \partial_m r \partial_i w_l \right)
    \end{split}
\end{align}
while $\tilde L_3$ is given by 
\begin{align}
    \begin{split}
    (\tilde{L}_3 w)_i &=  r^{-\frac{1}{\kappa}}
    G^{ml} \partial_l \left( r^{1+\frac{1}{\kappa}}
    (\partial_m w_i - \partial_i w_m ) \right)
    \end{split}
\end{align}
Then a direct computation shows that
\[
\begin{aligned}
r^{\frac{1}{\kappa}}((\tilde L_2+ \tilde L_3) w)_i 
= & \ \partial_l ( G^{ml} r^{1+\frac{1}{\kappa}} \partial_m w_i) + 
r^\frac1{\kappa} (\partial_l r G^{lm} \partial_m w_i - \partial_l G^{lm} \partial_i w_m) 
\end{aligned}
\]
We will take advantage of the covariant nature of this operator in order to 
simplify it. Interpreting $G$ as a dual metric and $w$ as a one form, we see that the above operator 
viewed as a map from one forms to one forms is invariant with respect to linear changes of coordinates. Here we are interested in changes of coordinates which preserve the 
surfaces $x_n = const$. But even with this limitation, it is possible to choose a linear change of coordinates, namely the semigeodesic coordinates relative to the surface $x_n = 0$,
\[
y' = Ax' + b x_n, \qquad y_n = x_n
\]
so that the metric $G$ becomes a multiple of the identity. Then the estimate \eqref{E:Coercive_estimate_w_short} reduces to its euclidean counterpart, which 
is discussed in detail in \cite{IT-norel} in the corresponding nonrelativistic context.
\end{proof}

To finish the proof of Theorem~\ref{T:Energy_norm_equivalence}, we will establish 
\begin{align}
\begin{split}
    \norm{(s_{2j-2},w_{2j-2})}_{\mathcal{H}^{2k-2j+2}}
    & \lesssim 
     \norm{(s_{2j},w_{2j})}_{\mathcal{H}^{2k-2j}}
           + \varepsilon
      \norm{(r,v)}_{\mathcal{H}^{2k}},
           \, 1 \leq j \leq k,
\end{split}
     \label{E:Concatenating_bounds}
\end{align}
where $\varepsilon>0$ is sufficiently small. We are using  $\varepsilon$ here to include 
two types of small error terms: (a) the terms that we estimate using $O(A)$ as well as
(b) the terms that have an extra factor of $r$ and for which we can
use smallness of $r$ near the boundary; the latter type arise
from the last term of \eqref{E:r_eq}. Concatenating these  estimates we then obtain the conclusion of the theorem.

To prove \eqref{E:Concatenating_bounds}, we first
consider 
$\norm{(F_{2j},G_{2j})}_{\mathcal{H}^{2k-2j}}$.
Using our interpolation inequalities, the non-endpoint
property, and the structure of 
$(F_{2j},G_{2j})$ described in 
in Lemma \ref{L:Recurrence_relations}, we obtain
\begin{align}
    \norm{(F_{2j},G_{2j})}_{\mathcal{H}^{2k-2j}}
    \lesssim \varepsilon \norm{(r,v)}_{\mathcal{H}^{2k}}.
    \nonumber
\end{align}

It remains to handle the term $\norm{(s_{2j},w_{2j})}_{\mathcal{H}^{2k-2j}}$. For 
$j=k$ the desired estimate is 
a direct consequence
of Lemma \ref{L:Elliptic_estimates}.

We move to treat the case $2 \leq j < k$. 
The idea is to apply Lemma \ref{L:Elliptic_estimates}
with $s_{2j-2}$ and $w_{2j-2}$ replaced by suitable
weighted derivatives of themselves. More precisely,
we set
\begin{equation*}
\left\{
    \begin{aligned}
    &s := L s_{2j-2}\\
    &w := L w_{2j-2},
    \nonumber
\end{aligned}
\right.
\end{equation*}
where 
\begin{align}
    L = r^a \partial^b, \quad 2a \leq b \leq 2(k-j).
    \nonumber
\end{align}
Applying $L$ to \eqref{E:Recurrence_s}, we obtain
\begin{align}
     L s_{2j} & = \tilde{L}_1  L s_{2j-2} + [L,\tilde{L}_1 ]
     L s_{2j-2}
     + L F_{2j}.
 \nonumber
\end{align}
The term $L F_{2j}$ can again be dealt with using 
Lemma \ref{L:Recurrence_relations}, as above. 
Thus we focus on
the commutator. To analyze it, we consider induction on $a$,
starting at $a=0$, and observe the following:
\begin{itemize}
    \item All terms where at least one $r$ factor gets differentiated
    twice are non-endpoint terms and can be estimated
    by interpolation.
    \item The terms where two $r$ factors are differentiated
    are handled by the induction on $a$.
    \item Terms where only one $r$ gets differentiated are also handled by induction on $a$ unless $a=0.$
\end{itemize}
Therefore, all terms in the commutator where $a>0$ are 
perturbative terms. We now focus on the case $a=0$.

Consider a frame $(x^\prime, x^n)$ in Minkowski space that is 
adapted to a point near the boundary in the sense that
\begin{align}
    |\partial^\prime r| \lesssim A, \quad 
    |\partial_n r - 1 | \lesssim A.
    \nonumber
\end{align}
Then, all terms in the commutator with tangential derivatives only
are error terms. For terms involving $\partial_n$, we find
\begin{align}
    \begin{split}
    [\partial_n^b,\tilde{L}_1] s & \approx 
    b a_2 \GInvP^{ij} \partial_i \partial_j \partial_n^{b-1}s 
    \\
    & \approx b a_2 \GInvP^{ij} \partial_i r \partial_j \partial_n^b s
    + b a_2 \GInvP^{n i^\prime} \partial_{i^\prime} \partial_n^b s
    + b a_2 \GInvP^{i^\prime j^\prime} \partial_{i^\prime} 
    \partial_{j^\prime} \partial_n^{b-1}s,
    \end{split}
    \nonumber
\end{align}
where primed indices run from $1$ to $n-1$. The last two terms 
on the RHS can
be treated by yet another induction, this time over $b$. The first term on the RHS can be combined back with $\tilde{L}_1$, yielding
$\partial_n^b \tilde{L}_1 \approx \tilde{L}^b_1 \partial_n^b$, where
\begin{align}
    \tilde{L}_1^b = 
    r a_2 \GInvP^{ij} \partial_i  \partial_j s
+ a_2 \left(\frac{1}{\kappa} + b \right) \GInvP^{ij} \partial_i r \partial_j s.
\nonumber
\end{align}
The operator $\tilde{L}_1^b$ has a similar structure to $\tilde{L}_1$, and an inspection in the proof of Lemma \ref{L:Elliptic_estimates}
shows that the corresponding coercive estimate for $s$ holds with 
$\tilde{L}_1^b$ in place of $\tilde{L}_1$.

The above argument works for $j\geq 2$ in that \eqref{E:Recurrence_s}
is valid only for $j\geq 2$. However, a minor change in the above
using the definition $s_2$ yields the result also for $j=1$. This
takes care of the $s$ terms in \eqref{E:Concatenating_bounds};
the proof for the $w$ terms is similar.
\end{proof}

\subsection{Energy estimates}
In this Section we establish the following.

\begin{theorem}
The energy functional $E^{2k}$ defined in \eqref{E:Total_energy_def} satisfies the following estimate:
\begin{align}
\frac{d}{dt} E^{2k}(r,v) \lesssim_A B \norm{ (r,v) }_{\mathcal{H}^{2k}}^2.
\nonumber
\end{align}
\end{theorem}

\begin{proof}
In view of equations \eqref{E:s_eq}-\eqref{E:w_eq}
and \eqref{E:Transport_vorticity_2k}, the the energy estimates for the linearized equation in Section \ref{S:Linearized}, and estimates for transport equations,
it suffices to show that
the terms $f_{2k}$, $g_{2k}$ and $h_{2k}$, given by \eqref{E:Source_s_eq}, \eqref{E:Source_w_eq}, 
and \eqref{E:Source_vort_2k_eq}, respectively, 
are perturbative, i.e., they satisfy the estimate
\begin{align}
\norm{ (f_{2k}, g_{2k} ) }_{\mathcal{H}}
+ \norm{ h_{2k}}_{L^2( r^\frac{1}{\kappa})}
\lesssim B \norm{(r,v)}_{\mathcal{H}^{2k}}.
\nonumber
\end{align}
To prove this bound we need to understand the structure of $(f_{2k},g_{2k})$, respectively $h_{2k}$:
\begin{lemma}\label{l:fg-2k}
Let $k \geq 1$. Then source terms $f_{2k}$ and $g_{2k}$ in the linearized equations \eqref{E:Good_variables_eq} for $(s_{2k},w_{2k})$, given by \eqref{E:Source_s_eq}-\eqref{E:Source_w_eq}
are multilinear expressions in $(r,\nabla v)$, with coefficients which are smooth functions
of $(r,v)$,  which have order $\leq k-\frac12$, respectively $\leq k$, with exactly $2k+1$ derivatives,
and which are not endpoint, in the sense  that there is no single factor in $f_{2k}$,
respectively $g_{2k}$ which has order larger that $k-1$, respectively $k-\frac12$.

Similarly, the source term $h_{2k}$  in the vorticity transport equation \eqref{E:Transport_vorticity_2k},
given by \eqref{E:Source_vort_2k_eq}, has the same properties as $g_{2k}$ above.
\end{lemma}

Once the lemma is proved, arguing similarly to Section \ref{S:Energy_coercivity}, we see that this
suffices to apply our interpolation results in Propositions~\ref{p:interpolation-g}, \ref{p:interpolation-c}
and \ref{p:interpolation-d} and obtain the desired bound. Here we remark that a scaling analysis 
shows that in the interpolation estimates we need to use at most one $B$ control norm, with 
equality exactly in the case of terms of highest order. One should also compare with the situation
in the similar computation in \cite{IT-norel}, where no lower order terms appear. Hence, the poof of the theorem is concluded once we prove the above lemma.

\begin{proof}[Proof of Lemma~\ref{l:fg-2k}]
Consider first $f_{2k}$. The fact that all terms in $f_k$ have order at most $k-\frac12$ is obvious.
The non-endpoint property can be understood as asking that there are no derivatives of order $2k+1$,
and that, in addition, for the terms of maximum order, they have at least two factors of the form
$\partial^{2+} r$ or $\partial^{1+}v$. Notably, this excludes any terms of the form 
\[
 f(r,v) r^{k+1-j}(\nabla r)^j  \partial^{2k+1-j} v, \qquad 0 \leq j \leq k+1. 
\]
A similar reasoning applies for $g_{2k}$ and $h_{2k}$, where the forbidden terms 
are those with a factor with $2k+1$ derivatives, as well as those of maximum order 
of the form
\[
f(r,v) r^{k-j}(\nabla r)^j  \partial^{2k+1-j} r, \qquad 0 \leq j \leq k. 
\]

We start with a simple observation, which is that, if in \eqref{E:Source_s_eq} or \eqref{E:Source_w_eq}, any derivative falls on a coefficient such as $\GInv$, $a_0$, $a_1$, or $a_2$, then we obtain lower order terms
which automatically satisfy the above criteria. Thus, for the purpose of this Lemma we can treat
these coefficients as constants.

A second observation is that there are no factors with $2k+1$ derivatives in either $s_{2k}$
or $w_{2k}$, due to the commutator structures present in both \eqref{E:Source_s_eq} or \eqref{E:Source_w_eq}. This directly allows us to discard all lower order terms, and in particular those containing $a_1$ and $a_3$. By the same token we  can set $a_0=1$ and $\br = 1$.

Given the above observations, it suffices to consider the reduced expressions
\begin{align}
\label{E:Source_s_eq-red}
&f_{2k}^{reduced} =  \ \GInvP^{ij} [ r  \partial_i, \Dt^{2k} ] v_j
-\frac{1}{\kappa} \GInvP^{ij}  \Dt \left( \partial_i r \right) \Dt^{2k-1} v_j 
\\
\label{E:Source_w_eq-red}
&(g_{2k}^{reduced})_i  =  \ a_2( \Dt^{2k-1} [ \partial_i, \Dt] r - 
\frac{a_0}{\kappa\br }\GInvP^{jl} \partial_j r \partial_i \Dt^{2k-1} v_l)
\\&  \nonumber
\qquad \qquad \quad + a_2(
[\partial_i, \Dt^{2k-1} ] \Dt r -  \frac{1}{\kappa} \GInvP^{ml} \partial_i \partial_m  r  \Dt^{2k-1} v_l),
\end{align}

Consider $f_{2k}^{reduced}$ first.
When commuting $\partial$ and $\Dt^{2k}$,
this produces at least one $\partial v$, so $[r \partial_i, \Dt^{2k} ] v_j$
is not an endpoint term. Similarly, $\Dt \left( \partial_i r \right)$ has order $1/2$
so the second expression is also not an endpoint term.

We now investigate $g_{2k}^{reduced}$. Neither of the first two terms is 
perturbative, but we have a leading order cancellation between them,
based on the relations
\begin{align}
[\Dt, \partial_i]= - \partial_i \left( \frac{v^j}{v^0} \right) \partial_j,
\nonumber
\end{align}
and
\begin{align}
\partial_i \left( \frac{v^j}{v^0} \right) = 
\frac{a_0}{\kappa \br } 
\GInvP^{jl} \partial_i v_l 
- \frac{\br^{1+\frac{2}{\kappa}}}{(v^0)^3} v^j  \partial_i r.
\label{E:coefficient_commutator_Dt_partial}
\end{align}

The contribution of the second term is lower order and thus perturbative. The 
contribution of the first term is combined with the second term in \eqref{E:Source_w_eq-red}
to obtain a commutator structure
\[
\left[\Dt^{2k-1}, \frac{a_0}{\kappa \br}  \GInvP^{jl} \partial_j r \partial_i \right] v_l,
\]
which yields only balanced terms. 

The third term in \eqref{E:Source_w_eq-red} is also balanced due to the commutator 
structure, while the last term has a direct good factorization.

\bigskip

We next move to $h_{2k}$. From \eqref{E:Source_vort_2k_eq}
we see that we are commuting $r^a \partial^b$ with either
$\Dt$ or $\partial v$, so we always obtain $\partial v$
factors that give non-endpoint terms. The only possible exception
is when all derivatives in the commutator with $\Dt$ are applied to the $r$ term in $v^0$. 
But this yields a lower order term.
\end{proof}
\end{proof}

\section{Construction of regular solutions\label{S:Regular_solutions}}

In this section we provide the first step in our
proof of local well-posedness, namely, here 
we present a constructive proof of regular 
solutions.  The rough solutions are obtained 
in the last section as unique limits of regular solutions.

Given an initial data $(\mathring{r},\mathring{v})$ with regularity 
\[
( \mathring{r},\mathring{v}) \in \bfH^{2k},
\]
where $k$ is assumed to be sufficiently large, we will construct 
a local in time solution, bounded in  $\bfH^{2k}$,
with a lifespan depending on the $\bfH^{2k}$ size of the data.

\subsection{Construction of approximate solutions}
We discretize the problem with a time-step  $\epsilon > 0$.
Then, given an  initial data $( \mathring{r},\mathring{v}) \in \bfH^{2k}$,
our objective is to  produce a discrete approximate 
solution $(r(j\epsilon), v(j\epsilon))$, with properties as follows:

\begin{itemize}
\item (Norm bound)  We have
\begin{equation}
\nonumber 
E^{2k}(r((j+1)\epsilon),v ((j+1)\epsilon)) 
\leq (1+ C \epsilon) E^{2k}(r((j \epsilon),v (j\epsilon)) .
\end{equation}

\item (Approximate solution) 
\begin{equation}
\nonumber 
\left\{
\begin{aligned}
& r((j+1)\epsilon) - r(j \epsilon)  + \epsilon \left[
v^m \partial_m r +
r \GInvP^{ml} \partial_m v_l  + r a_1 v^l \partial_l r \right](j\epsilon)= O(\epsilon^{2}) \\
&v_i((j+1)\epsilon) - v_i(j \epsilon)    
+ \epsilon \left[ v^m \partial_m v_i+
a_2 \partial_i r\right] (j\epsilon)= O(\epsilon^{2}). 
\end{aligned}
\right.
\end{equation}
\end{itemize}
The first property will ensure a uniform energy bound for our sequence.
The second property will guarantee that in the limit we obtain an exact solution.
There we use a weaker topology, where the exact choice of norms is not so important (e.g. $C^2$).

Having such a sequence of approximate solutions, it is straightforward to  produce, as the   limit on a subsequence, an exact solution $(r,v)$ on a short time interval which stays bounded in the above topology. 
The key point is the construction of the above sequence. It suffices to carry out a single step:

\begin{theorem}\label{t:onestep}
Let $k$ be a large enough integer. Let $(\rz,\vz) \in \bfH^{2k}$ with size 
\[
E^{2k}(\rz,\vz) \leq M,
\]
and $\epsilon \ll_M 1$. Then there exists a one step iterate $(\ro,\vo)$ with the following properties:

\begin{enumerate}
\item (Norm bound)  We have
\begin{equation}
  \nonumber 
 E^{2k}(\ro,\vo) \leq (1+ C(M) \epsilon) E^{2k} (\rz,\vz),
\end{equation}

\item (Approximate solution) 
\begin{equation}
\nonumber 
\left\{
\begin{aligned}
& \ro - \rz  + \epsilon[
\vz^i \partial_i r + 
\rz \GInvz^{ij} \partial_i \vz_j  + \rz \mathring{a}_1 \vz^i \partial_i \rz] = O(\epsilon^{2}) 
\\
& \vo_i - \vz_i  + \epsilon [
\vz^j \partial_j \vz_i +
\mathring{a}_2 \partial_i \rz] = O(\epsilon^{2}) ,
\end{aligned}
\right.
\end{equation}
where $\GInvz$, $\mathring{a}_1$, and $\mathring{a}_2$ 
are $\GInv$, $a_1$, and $a_2$ evaluated at $(\rz,\vz)$.
\end{enumerate}
\end{theorem}

The strategy for the proof of the theorem is the same as in
the last two authors' previous paper \cite{IT-norel}, by splitting the time step into three:
\begin{itemize}
    \item Regularization,
    \item Transport,
    \item Newton's method,
\end{itemize}
where the role of the first two steps is to improve the error estimate in the third step. The regularization step is summarized in the next Proposition:

\begin{proposition}\label{p:reg-step}
Given  $(\rz,\vz) \in \bfH^{2k}$,  there exist  regularized versions $(r,v)$
with the following properties:
\begin{equation}
\nonumber 
r-\rz = O(\epsilon^{2}), \qquad v - \vz = O(\epsilon^{2}),
\end{equation}
respectively
\begin{equation}
\nonumber
E^{2k} (r,v)  \leq (1+ C \epsilon) E^{2k} (\rz,\vz) ,
\end{equation}
and
\begin{equation}
\nonumber 
\| (r,v)\|_{\H^{2k+2}}   \lesssim \epsilon^{-1}M.
\end{equation}
\end{proposition}

\begin{proof}
We repeat the construction in \cite{IT-norel}. There are
only a few minor differences, namely:
\begin{itemize}
    \item The self-adjoint operators $L_1$, $L_2$ and $L_3$ there are replaced by their counterparts in this paper, i.e., \eqref{E:L_1_def_not_symmetric}, \eqref{E:L_2_def},
    and \eqref{E:L_3_tilde_def} (recall that $L_1=\hat{L}_1$ and
    $L_3 = \tilde{L}_3$). 
    \item Using \eqref{E:L_hat_no_hat_relation}, 
    relations similar to 
    \eqref{E:Recurrence_equations} continue to 
    hold for the self-adjoint operators. Thus, the approximate
    relations between $(s_{2k},w_{2k})$ and
    $(s^-_{2k},w^-_{2k})$ in Section 6 of \cite{IT-norel} also hold here.
    \item The elliptic estimates of Lemma \ref{L:Elliptic_estimates} hold for $L_1$ and $L_2,L_3$, 
    with essentially the same proof.
\end{itemize}
Aside from the above minor differences, the most important
observation in invoking the proof given 
in \cite{IT-norel} 
is that the counterpart of Lemma 6.3 in \cite{IT-norel} 
still holds with a minor change. For convenience we state here its counterpart
(below, $D s_{2k}$ and 
$D w_{2k}$ are the differentials
of $s_{2k}$ and $w_{2k}$ as functions
of $r$ and $v$):

\begin{lemma}\label{l:Dsw}
We have the algebraic relations
\begin{equation} 
\nonumber 
\left\{
\begin{aligned}
&Ds_{2k}(\ro,\vo) (\rz-\ro,\vz-\vo) \, = \ \ (L_{1}(\ro))^k (\rz-\ro) + \tilde F_{2k}
\\
&Dw_{2k}(\ro,\vo) (\rz-\ro,\vz-\vo) =  \  (L_{2}(\ro))^k (\vz-\vo) + 
\tilde G_{2k},    
\end{aligned}
\right.
\end{equation}
where the error terms $(\tilde F_{2k},\tilde G_{2k})$ are linear in 
$(\rz-\ro,\vz-\vo)$,
\[
\tilde F_{2k} = D^1_{2k}(\ro,\vo)(\rz-\ro,\rz-\vo),
\qquad \tilde G_{2k} = D^2_{2k}(\ro,\vo)(\rz-\ro,\rz-\vo).
\]
Their coefficients are multilinear differential expressions in 
$(\ro,\vo)$, have order at most $k-1$, respectively $k-\frac12$, 
and whose monomials fall into one of the following two classes:
\begin{enumerate}[label=\roman*)]
\item Have maximal order but contain at least one factor with order $> 0$,
i.e. $\partial^{2+} \ro$ or $\partial^{1+}\vo $, or
\item Have order strictly below maximum.
\end{enumerate}
\end{lemma}

By comparison, the similar relations in Lemma 6.3 in \cite{IT-norel}
are homogeneous, so only terms of type (i) arise in the error terms.
Here our equations are no longer homogeneous, and lower order terms do
appear. In particular, we note that all the contributions 
coming from the last term  in the first equation \eqref{E:r_eq}
belong to the class (ii) above. This is correlated with and motivates the fact that this term was neglected in our definition of the operator $L_1$.

With these observations in mind, the proof given in 
\cite{IT-norel} applies directly here.
\end{proof}

We now use Proposition~\ref{p:reg-step} in order to prove Theorem~\ref{t:onestep}. 

\medskip

\noindent \emph{Proof of Theorem~\ref{t:onestep}.}
For the transport step, we define 
\begin{equation}
\nonumber 
\xo^i = x^i + \varepsilon \frac{v^i(x)}{v^0(x)},
\end{equation}
where, in agreement with the our definition of the
material derivative, we iterate the coordinates by flowing with $v^i/v^0$, and not simply $v^i$.

Then we carry out the Newton step, and define $(ro,vo)$ by
\begin{equation*}
\left\{
\begin{aligned}
&\ro(\xo)  = r(x)
-\varepsilon\Big[   r \GInvP^{ij} \partial_i v_j  + r a_1 v^i \partial_i r \Big ](x)
\\
&\vo_i(\xo)  = v_i(x) - \varepsilon \left[ a_2 \partial_i r
\right] (x).
\end{aligned}
\right.
\end{equation*}

To show that $(\ro,\vo)$ have the properties in the Theorem, the argument is 
completely identical to the one in \cite{IT-norel}.
\qed 

\section{Rough solutions and continuous dependence\label{S:Rough_solutions_and_continuation}} 
The last task of the current work is to construct rough solutions as limits of smooth solutions, and conclude the proof of Theorem~\ref{T:LWP}.  Fortunately, the arguments in the preceding paper \cite{IT-norel}
by the last two authors for the similar part of the results apply word for word. This is despite the fact there are several differences between the  two problems that play a role on how the  energy estimates are obtained, as well as on how  uniqueness is proved.   However, the functional framework developed in \cite{IT-norel} and also implemented here does not see these differences. Furthermore,  the proof of the similar result in \cite{IT-norel} only uses (i) the regularization procedure in Section~\ref{S:Function_spaces}, (ii) the difference bounds of Theorem~\ref{T:Uniqueness}, and (iii) the energy estimates of Theorem ~\ref{T:Energy_estimates}, without any reference to their proof.

Thus, in our current result we rely on the same  succession of steps  as in the non-relativistic companion work of the last two authors \cite{IT-norel}, which we briefly outline here for the reader. These steps are:

\medskip

\noindent \emph{1. Regularization of the initial data.} We regularize the initial data; this is achieved by considering a family of dyadic regularizations of the initial data as described in Section~\ref{S:Function_spaces}. These data generate corresponding smooth solutions by Theorem~\ref{T:LWP}.  For these smooth solutions we control on the one hand higher Sobolev norms $\mathcal{H}^{2k+2j}$
using our energy estimates in Theorem~\ref{T:Energy_estimates}, and on the other hand the $L^2$-type distance between consecutive 
solutions,  which  is  at  the  level  of  the $\mathcal{H}$ norms, by Theorem~\ref{T:Uniqueness}. 

\medskip

\noindent \emph{2. Uniform bounds for the regularized solutions.}  To prove these bounds  we use a bootstrap argument on our control norm $B$, where $B$ is time dependent. The need for an argument of this kind is obvious. Once we have the regularized data sets $(\rz^h,\vz^h)$, we also have the corresponding smooth solutions  $(r^h,v^h)$ generated by the smooth data $(\rz^h,\vz^h)$. A-priori these solutions exist on a time interval that  depends on $h$. Instead, we would like to have a lifespan bound which is independent of $h$. This step  requires closing the bootstrap argument via the  energy estimates already obtained in Section~\ref{S:Energy_estimates}.

\medskip

\noindent \emph{3. Convergence of the regularized solutions.}  We obtain the convergence of the regular solutions $(r^h,v^h)$ to the rough solution $(r,v)$ by combining  the  high  and  the  low  regularity bounds  directly. This yields  rapid  convergence  in  all $\mathbf{H}^{2k'}$ spaces  below  the  desired  threshold, i.e.   for $k'<k $. Here we  rely primarily on results in Section~\ref{S:Uniqueness}, namely Theorem~\ref{T:Uniqueness}.

\medskip

\noindent \emph{4. Strong convergence.} Here we prove the convergence of the smooth solutions to the rough limit in the strong topology $\bfH^{2k}$. To  gain  strong  convergence   in $\bfH^{2k}$ we  use  frequency  envelopes   to  more accurately control both the low and the high Sobolev norms above.
This allows us to bound differences in the strong $\bfH^{2k}$ topology.  
A similar argument yields continuous dependence of the solutions in terms of the initial data, also in the strong topology. For more details we refer the reader to \cite{IT-norel}.

\bigskip

\bibliography{euler.bib}

\end{document}